\documentclass[11pt]{amsart}

\usepackage{amsmath,amssymb,amsfonts,amsthm,graphicx}

\usepackage{pinlabel}

\usepackage[usenames,dvipsnames]{color}

\usepackage[usenames,dvipsnames]{color}

\newcommand{\hul}{h_{u,l}}

\newtheorem{thm}{Theorem}[section]
\newtheorem{prop}[thm]{Proposition}

\newtheorem{dfn}[thm]{Definition}

\numberwithin{equation}{section}
\newcommand{\Z}{\mathbb{Z}}

\newcommand{\R}{\mathbb{R}}
\newcommand{\A}{\mathcal{A}}

\newcommand{\e}{\epsilon}

\newtheorem{theorem}{Theorem}[section]
\newtheorem{lemma}{Lemma}[section]
\newtheorem{corollary}{Corollary}[section]
\newtheorem{definition}{Definition}[section]
\newtheorem{proposition}{Proposition}[section]

\newtheorem{ax}[thm]{Axiom}

\theoremstyle{definition}

\newtheorem{example}[theorem]{Example}
\newtheorem{remark}[theorem]{Remark}
\newtheorem{convention}[theorem]{Convention} 
\newtheorem{observation}[theorem]{Observation}

\begin{document}
\title[Generating families and augmentations]{Generating families and augmentations for Legendrian surfaces}

\thanks{ The first author is partially supported by grant 429536 from the Simons Foundation.  The second author is partially supported by grant 317469 from the Simons Foundation. He thanks the Centre de Recherches Mathematiques for hosting him while some of this work was done.}

\author{Dan Rutherford}

\author{Michael Sullivan}

\begin{abstract}  We study augmentations of a Legendrian surface $L$ in the $1$-jet space, $J^1M$, of a surface $M$.   We introduce two types of algebraic/combinatorial structures related to the front projection of $L$ that we call chain homotopy diagrams (CHDs) and Morse complex $2$-families (MC2Fs), and show that the existence of either a $\rho$-graded CHD or MC2F is equivalent to the existence of a $\rho$-graded augmentation of the Legendrian contact homology DGA to $\Z/2$. A CHD is an assignment of chain complexes, chain maps, and homotopy operators to the $0$-, $1$-, and $2$-cells of a compatible polygonal decomposition of the base projection of $L$ with restrictions arising from the front projection of $L$.  An MC2F consists of a collection of formal handleslide sets and chain complexes, subject to axioms based on the behavior of Morse complexes in $2$-parameter families.  We prove that if a Legendrian surface has a tame at infinity  generating family, then it has a $0$-graded MC2F and hence a  $0$-graded augmentation.  In addition, continuation maps and a monodromy representation of $\pi_1(M)$ are associated to augmentations, and then used to provide more refined obstructions to the existence of generating families that  (i) are linear at infinity or (ii) have trival bundle domain.  We apply our methods in several examples.
\end{abstract}

\maketitle

{\small \tableofcontents}

\section{Introduction}

Pseudo-holomorphic curve based techniques 
have been used to prove
many results in contact and symplectic geometry over the last three decades.
One such method, which has enjoyed recent success in proving 
  rigidity results for Legendrian submanifolds and their exact Lagrangian cobordisms, is to package an appropriate class of pseudo-holomorphic curves into an invariant called Legendrian contact homology (LCH) 
which is the homology of a differential graded algebra (DGA).  
One way to extract information about a Legendrian using its LCH DGA is by considering augmentations which are DGA homomorphisms into a ground ring, that we take to be $\Z/2$ in this article.
As observed in \cite{Chekanov02}, an augmentation allows one to form a linearization of LCH which is more manageable than the full DGA.
Augmentations can arise geometrically from exact Lagrangian fillings (null-cobordisms) of a Legendrian, in which case their linearized homologies reflect the usual (relative) homology of the fillings \cite{Ekholm2, Rizell}.
In addition, augmentations of particular Legendrian surfaces have been used to provide
 powerful topological knot invariants through knot contact homology, with ties to string theory \cite{Ng, EENS, AENV}.
However, not all Legendrians have augmentations.



For a 1-dimensional Legendrian knot, $L$, in standard contact $\R^3 = J^1\R$ the existence problem for augmentations of the LCH DGA is well understood.  Fuchs found in \cite{Fuchs} an interesting combinatorial structure for a front projection called a normal ruling whose existence is equivalent to the existence of an augmentation, cf. \cite{FI,Sab}.  In addition, the existence of a $0$-graded normal ruling, (so also a $0$-graded augmentation), is equivalent to the existence of a linear at infinity generating family for $L$; see \cite{ChP,FuchsRutherford}.  Here, a generating family is a family of functions whose critical values coincide with the front projection of $L$.  To make this connection between generating families and augmentations more precise, Henry introduced an algebraic approximation for a generating family called a {\it Morse complex sequence}, and established a bijection between suitable equivalence classes of Morse complex sequences and homotopy classes of augmentations, \cite{Henry2011, HenryR}.

In this article, we take up analogous problems for Legendrian surfaces in $1$-jet spaces.  While a few important classes of Legendrian surfaces have had their DGAs extensively studied, eg. co-normal tori of braids/knots and isotopy spinnings of $1$-dimensional Legendrians, little has been known about the existence problem for augmentations of general Legendrian surfaces.  An obstacle to extending the methods used for $1$-dimensional Legendrians to the higher dimensional case has been the difficulty in $\dim \geq 2$ of giving an exact computation for the differential in the LCH DGA.  Building on work of Ekholm \cite{Ekholm07},  recent work of the authors \cite{RuSu1,RuSu2} gives explicit matrix formulas for the LCH differential of any generic Legendrian surface based on a choice of cellular decomposition for the base projection.   This cellular formulation of LCH is central to the present article.  

\subsection{Overview of results}

 Let $M$ be a surface, and let $L$ be a closed Legendrian surface in the $1$-jet space of $M$, $J^1M$.   
Given a  compatible cellular decomposition, $\mathcal{E}$, of the base projection to $M$ of $L$, the {\it cellular DGA} $(\mathcal{A}, \partial)$ of \cite{RuSu1} has a matrix of generators associated with each $0$-, $1$-, and $2$-cell of $\mathcal{E}$; see Section \ref{sec:Background} for details.  An augmentation $\epsilon: (\mathcal{A}, \partial) \rightarrow (\Z/2, 0)$, then produces scalar matrices assigned to each cell that can be profitably viewed as linear maps.  In Section \ref{sec:CHD}, by interpreting the augmentation equation $\epsilon \circ \partial =0$ from this point of view, we arrive in Proposition \ref{prop:AugmentationCHD} at an equivalent characterization of an augmentation as a {\it chain homotopy diagram} (abbr. CHD) that associates chain complexes to $0$-cells, chain maps to $1$-cells, and chain homotopy operators to $2$-cells, subject to certain conditions dictated by $L$.

To make contact with generating families, in Section \ref{sec:Morse2CHD} we introduce the notion of a {\it Morse complex $2$-family} (abbrv. MC2F) for $L$.  An MC2F is a collection of data associated to the front projection of $L$ that is modelled on the $2$-parameter family of Morse complexes that arises when $L$ has a generating family;  MC2Fs are the $2$-dimensional analog of the Morse complex sequences studied by Henry.   In particular, we show in Proposition \ref{prop:GFMorse2} that equipping a tame at infinity (see Section \ref{sec:GeneratingFamilies}) generating family for $L$ with an appropriate family of gradient-like vector fields produces an MC2F for $L$.

Our main results are summarized in the following:
\begin{thm}
\label{thm:Main2}  Let $M$ be a surface, $L \subset J^1M$ a closed Legendrian, and $\rho$ a divisor of the Maslov number, $m(L)$. 

The following conditions are equivalent:

\begin{enumerate}
\item The LCH DGA of $L$ has a $\rho$-graded augmentation to $\Z/2$.
\item $L$ has a $\rho$-graded chain homotopy diagram.
\item $L$ has a $\rho$-graded Morse complex $2$-family. 
\end{enumerate}

Moreover, if $L$ has a tame at infinity
 generating family, then the LCH DGA of $L$ has a $0$-graded augmentation to $\Z/2$.

\end{thm}




A generating family $F:E \rightarrow \R$ for a Legendrian in $J^1M$ has as its domain a fiber bundle over $M$, $\pi: E \rightarrow M$.  The bundle does not need to be trivial, and this can be reflected by a monodromy representation of the fundamental group of $M$ on the homology of a fiber  $E_{x_0} = \pi^{-1}(\{x_0\})$.  By carrying out a similar construction for MC2Fs, see Section \ref{sec:Comb}, and making use of the correspondence from Theorem \ref{thm:Main2}, in Section \ref{sec:MonoAug} we  associate  to an augmentation, $\epsilon$, and $x_0 \in M$ a {\it fiber homology} $H(\epsilon_{x_0})$ equipped with a {\it monodromy representation}  $\Phi_{\epsilon,x_0}: \pi_1(M,x_0)^{\mathit{op}} \rightarrow \mathit{GL}(H(\epsilon_{x_0}))$.  Using these representations, we provide in Proposition \ref{prop:ObstructA} obstructions to the existence of generating families that are (i) linear at infinity or (ii) defined on a trivial bundle.

In the concluding Section \ref{sec:Examples}, we illustrate our general results with several examples. An interesting  family of Legendrians, $L_\Gamma$, arising from $3$-valent graphs $\Gamma \subset M$ was  introduced by Treumann and Zaslow in \cite{TZ}.  Using Theorem \ref{thm:Main2} we show that $L_\Gamma$ has an augmentation if and only if the dual graph to $\Gamma$ is $3$-colorable;  this parallels a result from \cite{TZ} about constructible sheaves.   We also give examples to illustrate the obstructions from Proposition \ref{prop:ObstructA}.



We mention a few interesting directions for possible future study.

\begin{itemize}
 \item[(i)]  Currently, we do not know whether the statement about generating families in Theorem \ref{thm:Main2} can be strengthened to an if and only if statement.  A more precise question is whether every $0$-graded MC2F arises from an actual generating family via an appropriate choice of gradient vector field.  

\item[(ii)] The constructible sheaf invariants of Legendrian submanifolds introduced in \cite{STZ} have also been shown to have close ties to generating families, cf. \cite{Shende}, and (in dimension 1) augmentations \cite{NRSSZ}.  It is possible that the equivalent characterizations of augmentations for Legendrian surfaces from Theorem \ref{thm:Main2} could be useful for establishing a connection with sheaf-based invariants.  
\end{itemize}

\subsection{Organization}  The proof of Theorem \ref{thm:Main2} is based on the following logic: 
   \[
\mbox{Generating family} \rightarrow \mbox{MC2F} \leftrightarrow \mbox{CHD} \leftrightarrow \mbox{Augmentation}.
\]

In Section \ref{sec:Background} we review generating families, augmentations, and the cellular formulation of the LCH DGA for Legendrian surfaces from \cite{RuSu1, RuSu2}.  In Section \ref{sec:CHD}, we define CHDs and show that they are in bijection with augmentations of the cellular DGA.  
In Section \ref{sec:Morse2CHD}, we define MC2Fs and use the analysis of $2$-parameter families of functions from \cite{HatcherWagoner} to show how a generating family for $\Lambda$ produces an MC2F.  
In addition, given an MC2F, we associate continuation maps to paths in $M$.  The properties of these maps established in Proposition \ref{prop:ContinuationMap} allow us to define monodromy representations for MC2Fs and are later used for translating between CHDs and MC2Fs.
The construction of a CHD from an MC2F is carried out in Section \ref{sec:Proof}.  After establishing some tools that are useful for the construction of MC2Fs,  the converse construction of an MC2F from a CHD appears in Section \ref{sec:CHDMC2F}.
The monodromy representations for augmentations are constructed at the end of Section \ref{sec:CHDMC2F} with obstructions to particular types of generating families observed in Proposition \ref{prop:ObstructA}.  Finally, in Section \ref{sec:Examples} we apply our general results to several examples.

\section{Background}
\label{sec:Background}

\subsection{Legendrian surfaces}  
Let $M$ be a 2-dimensional manifold. Then 
$J^1M= T^*M \times \R_z$ is a 5-dimensional contact manifold with a standard contact structure 
$\xi = \mbox{ker}(dz - ydx)$ where $x = (x_1,x_2)$ are local coordinates for $M$ (which we denote sometimes by $M_x$) and $y = (y_1,y_2) \in T_xM$ are fiber coordinates.
A {\bf{Legendrian}} (surface) $L \subset J^1M$ is a two-dimensional submanifold such that $TL \subset \xi.$

Let $\Pi_F: J^1M\rightarrow J^0M= M \times \R_z$ be the so-called {\bf front projection}.
Let $\Pi_B: J^1 M \rightarrow  M$ be the {\bf base projection}.
We usually consider Legendrians that have generic front and base projections;  see \cite[Section 2.2]{RuSu1} for a detailed discussion.  Figure \ref{fig:Codim2Both} illustrates the generic singularities which arise in $\Pi_F(L)$ and $\Pi_B(L)$.   
At a swallowtail point, a pair of cusp edges and a crossing arc all meet.  We call a swallowtail point {\bf upward} (resp. {\bf downward}) if the sheet that connects the two cusp edges appears above (resp. below) the two crossing sheets.  
In the base projection, the image of the cusp edges divides a disk neighborhood of a swallowtail point into two parts.  We refer to the region between the two cusp edges, above which the cusp sheets exist, as the {\bf swallowtail region}.






\begin{figure}
\labellist
\small
\pinlabel $x_1$ [l] at 38 48
\pinlabel $x_2$ [b] at 2 86
\pinlabel $x_1$ [l] at 38 200
\pinlabel $x_2$ [l] at 22 226
\pinlabel $z$ [b] at 2 238
\endlabellist
\centerline{\includegraphics[scale=.4]{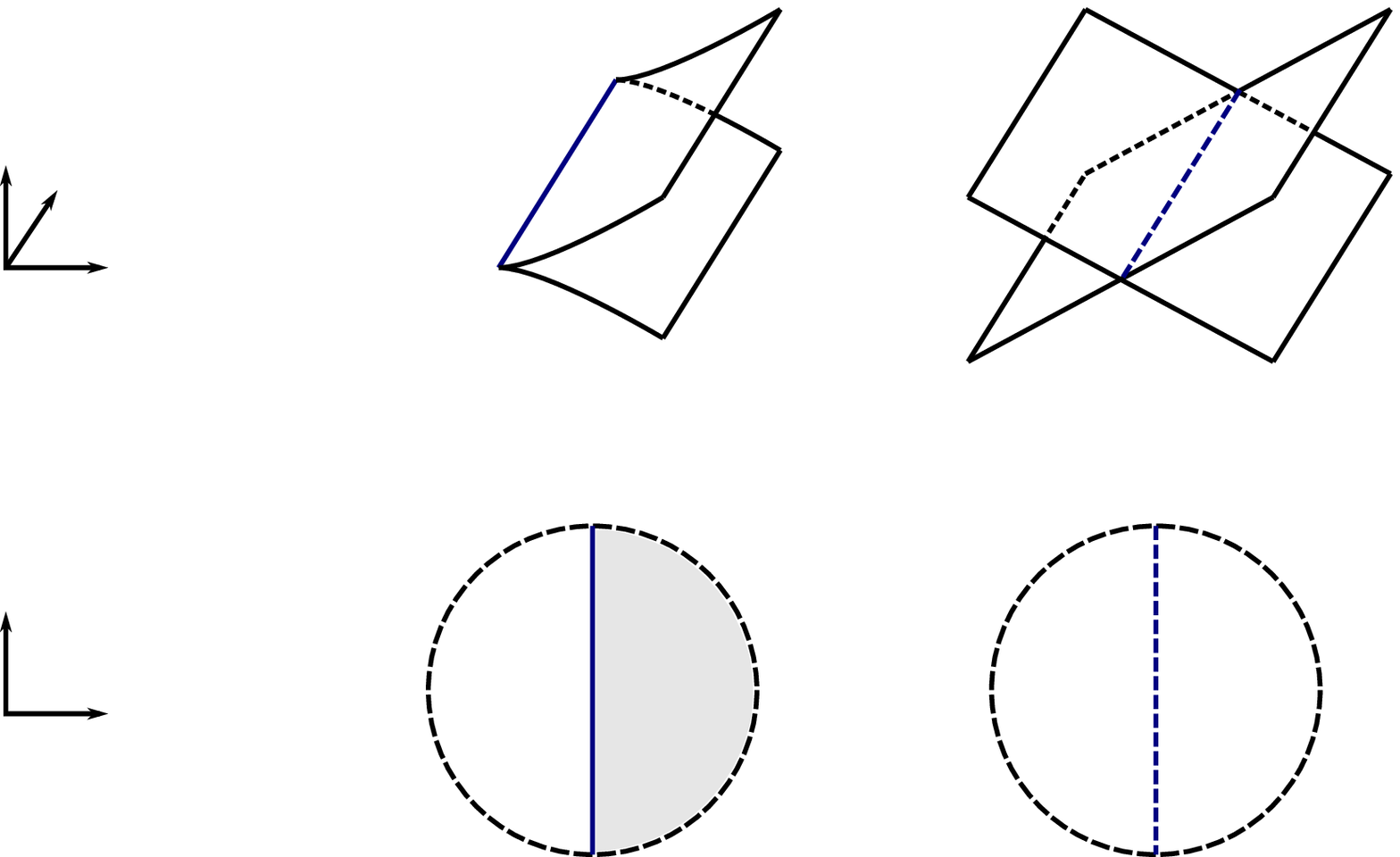} \quad \quad \includegraphics[scale=.4]{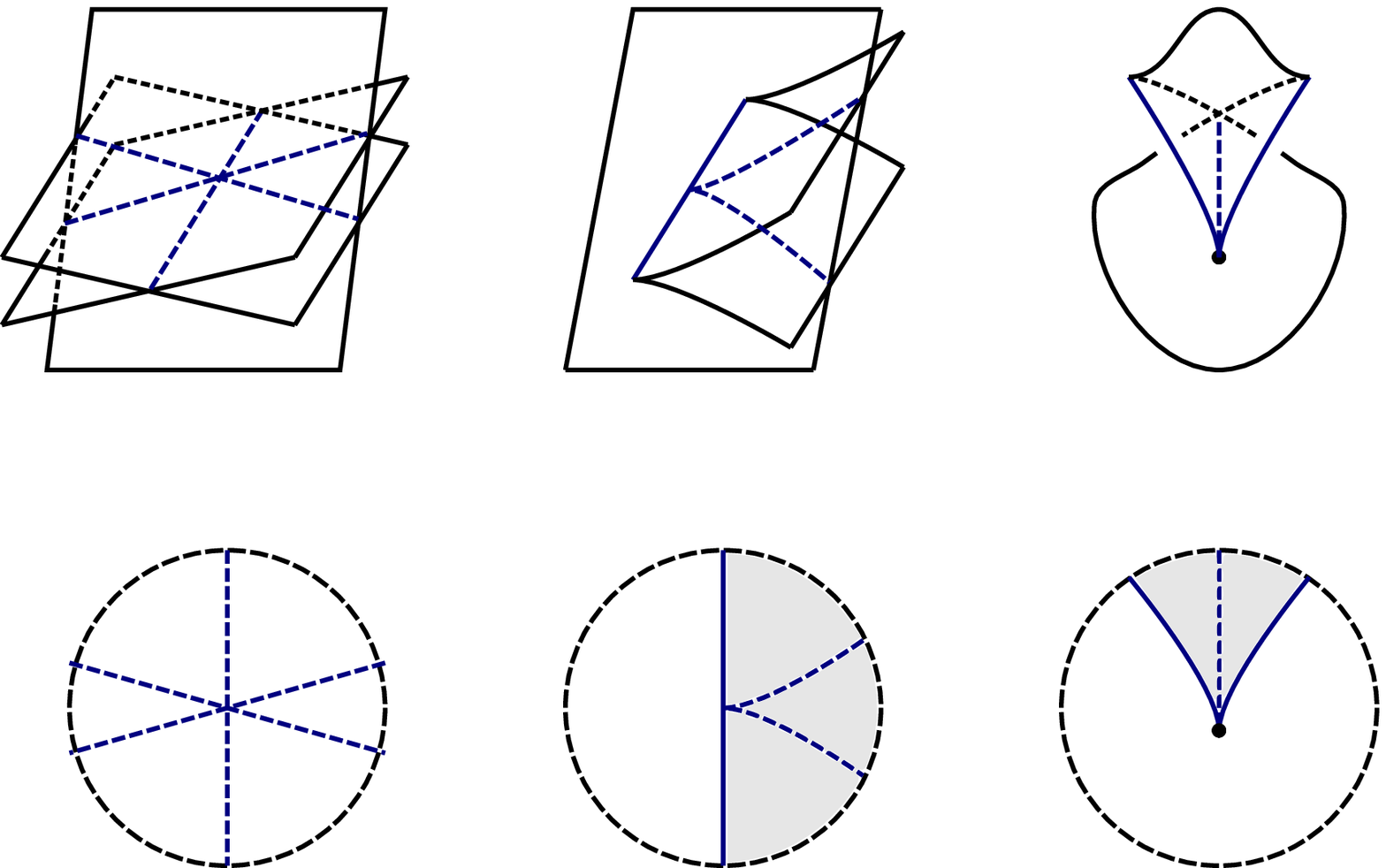}}

\caption{Generic singularities of front projections are pictured along with their base projections. From left to right:  cusps, crossings, triple points, cusp-sheet intersection, and swallowtail points.  Additional codimension $2$ singularities in the base projection arise as transverse intersections of two crossing/cusp arcs that are disjoint in $\Pi_{F}(L)$. }
\label{fig:Codim2Both}
\end{figure}

A generic loop $\gamma \subset L$ is assigned an integer $m(\gamma) = D(\gamma) - U(\gamma) \in \Z$ where $D(\gamma)$ (resp. $U(\gamma)$) are the number of times $\gamma$ crosses with a cusp edge of $L$ in the downward (resp. upward) direction.   This assignment gives a well defined cohomology class $m \in \mbox{Hom}(H_1(L;\Z), \Z) = H^1(L; \Z)$, and the {\bf Maslov number} of $L$, $m(L) \in \Z_{\geq0}$, is the non-negative generator of the image of $m$.   
A {\bf Maslov potential}, $\mu$, for $L$ is a locally constant function 
\[
\mu: L \setminus (\Sigma_{\mathit{cusp}} \cup \Sigma_{\mathit{st}})\rightarrow \Z/m(L),
\] where $\Sigma_{\mathit{cusp}} \cup \Sigma_{\mathit{st}} \subset L$ is the union of all cusp and swallowtail points, such that $\mu$ increases by $1$ when passing from the lower sheet to the upper sheet at any cusp edge.  Maslov potentials exist and, when $L$ is connected, are unique up to an overall additive constant.

\subsection{Generating families}
\label{sec:GeneratingFamilies}

We review generating families in the Legendrian setting; for more details and applications, see for example \cite{Chek, Traynor, SabTraynor}.    
Let $\pi: E \rightarrow M$ be a locally trivial fiber bundle over $M$ with manifold fiber $N$.
Given $F: E \rightarrow \R$ and $x \in M,$ we denote its restriction to a fiber by $f_x: \pi^{-1}(x) \cong N \rightarrow \R.$
We denote by $\eta = (\eta_1, \ldots, \eta_n) \in N$  locally-defined fiber coordinates and refer to a point in $E$ as $e = (x, \eta).$
  Suppose that $dF:E \rightarrow T^*E$ is transverse to the fiber normal bundle
\[
N_E = \left\{ (e, \nu) \in T^*E \,\, | \,\, \nu = 0 \,\, \mbox{on} \,\, \mbox{ker} (d\pi(e))\right\}.
\]
In coordinates, this is equivalent to $0$ being a regular value of $(x,\eta) \mapsto \partial_\eta F(x, \eta).$  This transversality condition ensures that the set of fiber critical points of $F,$
\[
\Sigma_F = \left\{ (x, \eta) \in E \,\, | \,\, (df_x)_\eta = 0 \right\} = (dF)^{-1}(N_E),
\]
is a manifold.  There is then a Legendrian immersion of $\Sigma_F$ into $J^1M$ given in coordinates by
\[
i_F: \Sigma_F \rightarrow J^1M, \quad  (x,\eta) \mapsto (x,y,z) = (x, \partial_xF(x,\eta), F(x,\eta)).
\]
When $i_F$ is an embedding with $i_F(\Sigma_F) = L$, we say that $F$ is a {\bf generating family} for $L$.  
If $F$ is a generating family for $L,$ then so too is $F \circ \phi$ where $\phi:E \rightarrow E$ is a fiber-preserving diffeomorphism.  In addition, {\bf stabilizations} of $F$, defined by $\overline{F}: E \times \R^{m} \rightarrow \R$, $\overline{F}(e, \mu) = F(e) + Q(\mu)$ for some non-degenerate quadratic form $Q: \R^{m} \rightarrow \R$, 
are also generating families for $L$.

In order to apply the tools of Morse theory to $F$, it is important to make some assumption about the behavior of $F$ outside of compact sets.  The following two  conditions are commonly used in the generating family literature.  
A generating family  $F: E \rightarrow \R$ is {\bf{linear at infinity}}  (resp. {\bf quadratic at infinity}) if $E= E' \times \R^k$ where $E'$ is a locally trivial fiber bundle with closed manifold fibers and, outside of a compact subset of $E$,  $F$ agrees with a fixed non-zero linear form (resp. a fixed non-degenerate quadratic form) on $\R^k$.  We say $F$ is {\bf tame at infinity} if $F$ is either linear or quadratic at infinity.
Note that in the linear at infinity case, the $\R^k$ factor must have $k \geq 1$, while $k=0$ is allowed in the quadratic at infinity case.  
If $M$ is non-compact, then a quadratic at infinity generating family cannot produce a compact Legendrian.

\begin{remark}
It can be shown that, after a fiber-preserving diffeomorphism, a stabilization of a linear (resp. quadratic) at infinity generating family can again be made linear (resp. quadratic) at infinity.
This is an important point for defining generating family homology invariants using tame at infinity generating families; see \cite{SabTraynor}.
\end{remark}

\subsection{Augmentations}

A {\bf{differential graded algebra}} (DGA) in this article is an associative graded unital 
algebra $\mathcal{A},$ equipped with a differential; that is, a derivation 
$\partial: \mathcal{A} \rightarrow \mathcal{A}$ which squares to 0 and decreases the grading by 1. 
We consider DGAs with  ground ring $\Z/2$, that are graded by  $\Z/m$ for some $m \in \Z_{\geq 0}$ (when $m =0$,  $\Z/m=\Z$).  
The DGAs we consider are freely generated by elements of homogeneous degree.

An {\bf{augmentation}} $\epsilon: (\mathcal{A}, \partial) \rightarrow (\Z/2, 0)$ is an
algebra morphism $\epsilon: \mathcal{A} \rightarrow \Z/2$ such that $\epsilon(1) =1$ and $\epsilon \circ \partial = 0.$
Given a divisor $ \rho \,|\, m$, we say that $\epsilon$ is {\bf $\rho$-graded} if $\rho$ preserves grading mod $m$.  Equivalently, if $\epsilon(a_i) \neq 0$ for a generator $a_i \in \mathcal{A}$ implies $|a_i| = 0 \mbox{ mod } \rho$.

In the context of Legendrian contact homology, the standard notion of equivalence used for DGAs is {\bf stable tame isomorphism} which also implies homotopy equivalence.  The existence of a $\rho$-graded augmentation is invariant under stable tame isomorphism, cf. \cite{Chekanov02} or \cite[Section 2.1.1]{RuSu1}.



\subsection{The Cellular DGA}  We refer the reader to \cite{EkholmEtnyreSullivan05b, EkholmEtnyreSullivan07} or other sources for the
pseudo-holomorphic based definition of the DGA underlying Legendrian contact homology (LCH).
Instead, for the remainder of this section we review the stable tame isomorphic Cellular DGA.  The Cellular DGA was introduced in \cite[Section 3]{RuSu1}, and proven to be stable-tame isomorphic to the usual LCH DGA in \cite{RuSu2}.

\begin{definition}  Let  $L \subset J^1M$ be a Legendrian surface with generic base projection.  A {\bf compatible polygonal decomposition} for $L$, $\mathcal{E}$, is a polygonal cell decomposition of $\Pi_B (L) \subset M$ that contains $\Pi_B(\Sigma)$ in its $1$-skeleton, and is equipped with 
\begin{enumerate}
\item A choice of orientation for each 1-cell. 
\item In the domain of each 2-cell, two of its 0-cell vertices are labeled as `initial' and `terminal' vertices $v_0, v_1.$
If $v_0 = v_1$ we must also choose a direction for the path around the circle from $v_0$ to $v_1$.
\item At each swallowtail point, we choose a labeling of the two corners that border the crossing locus.  One region is labeled $S$ and the other $T$.
\end{enumerate}
\end{definition}
\begin{convention}  \label{conv:1} In this article, to simplify the exposition, we will assume in addition that near swallowtail points, the $1$-skeleton of $\mathcal{E}$ agrees with the projection of the singular set with the three $1$-cells oriented away from the swallowtail point.   The cellular DGA can be defined without this assumption.  See Figure \ref{fig:LocalST}.
\end{convention}

Let $e^d_\alpha$ be a cell from $\mathcal{E}$ where $0 \leq d \leq 2$ is the dimension.
We let $L(e^d_\alpha)$ denote the {\bf set of sheets} of $L$ above $e^d_\alpha$.  This is defined as 
 the set of those connected components of $L$ above $e^d_\alpha$ that are {\it not contained in a cusp edge}, i.e.
\[
L(e^d_\alpha) = \pi_0\left(\Pi_B^{-1}(e^d_\alpha) \cap( L \setminus \Sigma_{\mathit{cusp}}) \right). 
\]
  Note that we do consider a swallowtail point above a $0$-cell to be a sheet.
  Each set $L(e^d_\alpha)= \{S^\alpha_p\}$ has a partial order by (point-wise) descending $z$-coordinate,
\[
S^\alpha_p \prec S^\alpha_q  \quad \Leftrightarrow \quad z(S^\alpha_p) > z(S^\alpha_q);
\]
two sheets are incomparable if and only if they meet at a crossing arc above $e^d_\alpha$ in $\pi_F(L)$.
When the sheets of $L(e^d_\alpha)$ are totally ordered by $z$-coordinates, we use $\{1, 2, 3, \ldots, n\}$ for the indexing set so that $S^\alpha_i \prec S^\alpha_{i+1}$.


The algebra $\mathcal{A}$ is freely generated as follows.  For each cell $e^d_\alpha$ we associate one generator for each pair of sheets $S^\alpha_p$, $S^\alpha_q \in L(e^d_\alpha)$ satisfying $S^\alpha_p \prec S^\alpha_q$.  We denote these generators as $a^\alpha_{p,q}$, $b^\alpha_{p,q}$, or $c^\alpha_{p,q}$ in the case where $e^d_\alpha$ is a $0$-cell, $1$-cell, or $2$-cell respectively. 
Sometimes we suppress the superscript $\alpha$ from notation.
The grading of $\mathcal{A}$ requires a choice of Maslov potential, $\mu$, and is defined on generators by
\begin{equation}
\notag
|c_{p,q}| = \mu(S_p) -\mu(S_q) +1;  \quad |b_{p,q}| = \mu(S_p) -\mu(S_q);  \quad \mbox{and  }|a_{p,q}| = \mu(S_p) -\mu(S_q) -1.
\end{equation} 

\subsubsection{The differential without swallowtail points}
In reviewing the differential, we start with the case that $L$ {\it does not have swallowtail points.}  We choose for each cell a bijection $\iota$ between $\{1, \ldots, n_\alpha\}$   and the indexing set for $L(e^d_\alpha)$  that is compatible with the partial ordering  of $L(e^d_\alpha)$ in the sense that 
\[
S_p \prec S_q  \quad \Rightarrow \iota(p) < \iota(q).
\]  
Using the bijection, we arrange the generators corresponding to $e^d_\alpha$ into a strictly upper triangular $n_\alpha \times n_\alpha$ matrix,
which we label $A,$ $B$ or $C$ accordingly.
Note that entries in the upper triangular part of $A$ or $B$ that correspond to pairs of sheets that cross are $0$.

Next, suppose that a cell $e^{d'}_{\beta}$ appears along the boundary of $e^d_\alpha$ with $d'<d$.  We then place the generators associated to $e^{d'}_\beta$ into a corresponding $n_\alpha \times n_\alpha$ {\bf boundary matrix} $X$ as follows:  Each sheet in $L(e^{d'}_\beta)$ belongs to the closure of a unique sheet in $L(e^d_\alpha)$.  This identifies the indexing set of $L(e^{d'}_{\beta})$ with a subset of $\{1, \ldots, n_\alpha\}$, and we place the generators associated to $e^{d'}_\beta$ into the corresponding rows and columns of $X$.  The remaining rows and columns correspond to sheets of $L(e^d_\alpha)$ that meet a cusp edge above $e^{d'}_{\beta}$, and such sheets come in pairs.  When $d'=0$ (resp. $d'=1$), we insert the $2\times 2$ block $\left[ \begin{array}{cc} 0 & 1 \\ 0 & 0 \end{array} \right]$ (resp. $\left[ \begin{array}{cc} 0 & 0 \\ 0 & 0 \end{array} \right]$) along the diagonal  
in the columns and rows that represent each cusping pair of sheets.  

For a 1-cell, let $A_+$  (resp. $A_-$) be the boundary matrices for the terminal (resp. initial) vertex.
For a 2-cell, let $A_{v_0}$ and $A_{v_1}$ be the boundary matrices associated  to the chosen initial and terminal vertices, $v_0$ and $v_1.$
In addition, let $B_1, \ldots, B_{j}$ and $B_{j+1}, \ldots, B_{m}$ denote the boundary matrices associated to the successive boundary edges that  appear in the domain of the characteristic map for the $2$-cell, as we travel the two paths along the boundary of $D^2$ from $v_0$ to $v_1$. 
(If $v_0=v_1$, then one of these paths is constant as specified in the definition of $\mathcal{E}$.) 
The differential $\partial:\A \rightarrow \A$ is then determined by the following matrix formulas where $\partial$ is applied entry-by-entry.
\begin{eqnarray}
\label{eq:differential} 
\partial A & = & A^2, \\
\notag
\partial B & = & A_+(I+B) + (I+B) A_-,\\
\notag
\partial C & = & A_{v_1}C + C A_{v_0} + (I+B_{j})^{\eta_j} \cdots (I+B_1)^{\eta_1} + (I+B_{m})^{\eta_m} \cdots (I+B_{j+1})^{\eta_{j+1}},
\end{eqnarray}
where $\eta_i \in \{-1,+1\}$ compares the orientation of the 1-cell with the orientation of the path from $v_0$ to $v_1$ on which it lies.
See Figure \ref{fig:Diff}.

\begin{figure}
\labellist
\small
\pinlabel $A=(a^{\alpha}_{i,j})$ [b] at 2 110
\pinlabel $A_-$ [b] at 100 110
\pinlabel $A_+$ [b] at 212 110
\pinlabel $B$ [b] at 148 114
\pinlabel $B_1$ [tr] at 338 36
\pinlabel $B_2$ [r] at 302 110
\pinlabel $B_3$ [br] at 342 156
\pinlabel $B_4$ [tl] at 450 68
\pinlabel $B_5$ [bl] at 422 154
\pinlabel $A_{v_0}$ [t] at 404 0
\pinlabel $A_{v_1}$ [b] at 390 186
\pinlabel $C$  at 378 94
\endlabellist
\centerline{\includegraphics[scale=.6]{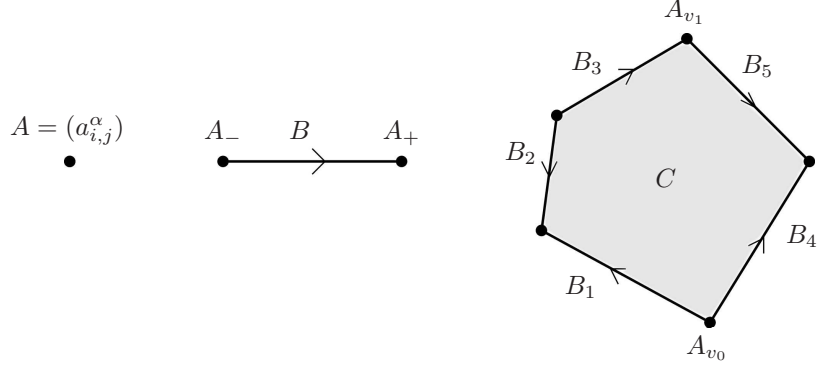}}

\quad

\caption{Differentials for the pictured cells are  $\partial A= A^2$; \quad $\partial B=A_+(I+B)+(I+B)A_-$; \quad $\partial C=A_{v_1}C+CA_{v_0}+(I+B_3)(I+B_2)^{-1}(I+B_1)+(I+B_5)^{-1}(I+B_4)$.}
\label{fig:Diff}
\end{figure}

\subsubsection{Adjustments for swallowtail points}  
In this article, we focus our arguments on the case of upward swallowtail points as pictured in Figure \ref{fig:Codim2Both}.
The downward swallowtail is similar;  for details see \cite[Sections 3.6-3.12]{RuSu1}.  
 Suppose  that near a swallowtail point, $e^0_{\mathit{st}}$, $L$ has $n$ sheets (resp. $(n-2)$ sheets)  inside (resp. outside) the swallowtail region, and  the sheets in position $k, k+1, k+2$ (with respect to descending $z$-coordinate) above the swallowtail region meet at the swallowtail point.  Recall that the two $2$-cell corners within the swallowtail region that border the crossing locus at the swallow tail point have been labeled with $S$ and $T.$ 

Let
\begin{equation}
\label{eq:DifferentialST}
\begin{array}{lll}
A_S & = [I + E_{k+2,k+1}] \widehat{A}_{k,k+2} [I + E_{k+2,k+1}];  &  A_T = [I+E_{k+1,k+2}]\widehat{A}_{k,k+1} [I+E_{k+1,k+2}];  \\ &  & \\
S & = I + \widehat{A}_{k,k+1} E_{k+2,k} + E_{k+1,k+2} &  T = I +E_{k+1,k+2}  \\
 & =I + \sum_{i<k} a_{i,k}E_{i,k} + E_{k+1,k+2};  &
\end{array}
\end{equation}
where $E_{i,j}$ is the matrix with all $0$'s except for a $1$ in the $(i,j)$-th entry, and $\widehat{A}_{i,j}$ is the $(n-2)\times(n-2)$ matrix $A$ over the swallowtail point enlarged by the $2\times 2$ block $\left[ \begin{array}{cc} 0 & 1 \\ 0 & 0 \end{array} \right]$ in columns (and rows) $i$ and $j.$ 

Let $B_{cr}$ denote the matrix over the 1-cell associated to the crossing locus with endpoint at $e^0_{st}$. 
If the ordering of the sheets used to form $B_{cr}$ agrees with that of the 2-cell marked by $S$ (resp. $T$) then in the differential
$\partial B_{cr}$ set the boundary matrix $A_{\pm}$ associated to $e^0_{st}$ equal to $A_S$ (resp. $A_T$).  By assumption on $\mathcal{E}$ in Convention \ref{conv:1}, all other $1$-cells with endpoints at $e^0_{st}$ have $n-2$ sheets, and we take the boundary matrix to just be $A$.   

For the 2-cell that includes the region marked by $S$ (resp. $T$), in equation (\ref{eq:differential}) we replace the $I+B_i$ factor   associated to the cusp edge that begins at the swallowtail point with the product $(I+B_i)S$  (resp. $(I+B_i)T$).


\section{Augmentations are CHDs}  \label{sec:CHD}

In this section, we examine augmentations of the cellular DGA.   By viewing the image of the matrices $A$, $B$, and $C$, as linear maps we establish in Proposition \ref{prop:AugmentationCHD} an equivalent characterization of augmentations as {\it Chain Homotopy Diagrams} which assign chain complexes, chain maps, and chain homotopies to the cells of $\mathcal{E}$. 

\subsection{Ordered complexes}

Let $V$ be a vector space over $\Z/2$ with specified basis $\mathcal{B}= \{v_p \,|\, p \in I\}.$
We use the inner product notation to denote the bilinear form $\langle v_p, v_q \rangle = \delta_{p,q}$, so that for $w = \sum \alpha_i v_i$  the $i$-th coefficient is  
$\alpha_i = \langle w, v_i \rangle \in \Z/2.$

\begin{dfn}
\label{dfn:UT}
Suppose that the basis $\mathcal{B}$ is equipped with a partial order $\prec$.  
A linear transformation $T: V \rightarrow V$ is {\bf{strictly upper triangular}} if 
\[
\langle T(v_q), v_p \rangle \neq 0 \quad \Rightarrow  \quad v_p \prec v_q.
\]
An {\bf ordered complex} is a triple $(V, \mathcal{B}, d)$ such that $d:V \rightarrow V$ is a differential, i.e. $d^2=0$, that is strictly upper triangular.
An ordered complex is {\bf $m$-graded} if basis vectors $v_p \in \mathcal{B}$ are assigned degrees $|v_p| \in \Z/m$, and  $d$ has degree $+1$ (mod $m$) with respect to the resulting grading on $V$. 
\end{dfn}

\subsection{Handleslide maps}

\begin{dfn}
\label{dfn:hsmap}
Let $V$ be a $\Z/2$-vector space with basis $\mathcal{B} = \{v_p \,|\, p \in I\}$.  Given $u,l \in I$, the {\bf{handle-slide map}} $h_{u,l}$ 
is the linear map satisfying
\begin{equation}
\label{eq:hsmap}
h_{u,l}(v_k) = v_k + \delta_{k,l} v_u.
\end{equation}
\end{dfn}

Note that since this article works with $\Z/2$-coefficients, $\hul^{-1} =\hul.$  When the indexing set $I$ is $\{1, \ldots, n\}$, the matrix for $h_{u,l}$ is $I+E_{u,l}$.

\subsection{Vector spaces associated to cells}

Let $L \subset J^1M$ be a Legendrian equipped with a Maslov potential $\mu$ and a compatible polygonal decomposition $\mathcal{E}$.  To each $d$-cell $e^d_\alpha \in \mathcal{E}$ we associate the vector space spanned by the (non-cusping) sheets of $L$ above $e^d_\alpha$,
\[
V(e^d_\alpha) = \mbox{Span}_{\Z/2} L(e^d_\alpha).
\]
Recall that $L(e^d_\alpha)$ is partially ordered by descending $z$-coordinate.  In addition, each $V(e^d_\alpha)$ has a $\Z/m(L)$-grading arising from 
\[
|S^\alpha_p| = \mu(S^\alpha_p).
\]

\subsection{Boundary differentials and maps} \label{sec:BoundaryDiff} In the following definitions we initially assume that $L$ has no swallowtail points, and then give modifications for the general case.

Suppose that a $0$-cell $e^{0}_\beta$ appears along the boundary of $e^d_\alpha$ with $d=1$ or $2$, and write 
  $e^0_\beta \stackrel{j}{\rightarrow} e^d_{\alpha}$ for a corresponding inclusion\footnote{There may be more than one such inclusion since $e^0_\beta$ may appear more than once along the boundary of $e^d_\alpha$.} of $e^0_\beta$ into the boundary of $D^d$, viewed as the domain of the characteristic map  $D^d \rightarrow \overline{e^d_\alpha} \subset M$.  Assuming that $V(e^{0}_\beta)$ has been given a differential $d_\beta$ such that $(V(e^{0}_\beta), L(e^{0}_\beta), d_\beta)$ is an ordered complex, we define a {\bf boundary differential}  
\[
\widehat{d}_\beta= \widehat{d}(e^{0}_\beta \stackrel{j}{\rightarrow} e^d_\alpha)
: V(e^d_\alpha) \rightarrow V(e^d_\alpha)
\] 
as follows.
The natural inclusion $i:L(e^{0}_\beta) \hookrightarrow L(e^d_\alpha)$ (where $i(S^\beta_p) = S^\alpha_q$ when $S^\beta_p \subset \overline{S^\alpha_q}$ in $L$), extends to an embedding $i:V(e^{0}_\beta) \hookrightarrow V(e^d_\alpha)$.  We have 
\[
V(e^d_\alpha)= i(V(e^{0}_\beta)) \oplus V_{\mathit{cusp}}
\]
where $V_{\mathit{cusp}}$ is spanned by the (possibly zero)
 sheets that meet a cusp edge above $e^{0}_\beta$.  We define $\widehat{d}_\beta$ to satisfy
\[
\widehat{d}_\beta= d_\beta \oplus d_{\mathit{cusp}}
\]
where $d_\mathit{cusp}(S^\alpha_b) = S^\alpha_a$ when sheets $S^\alpha_b$ and $S^\alpha_a$ meet at a cusp edge above $e^{0}_\beta$ with $S^\alpha_a$ (resp.  $S^\alpha_b$) the upper (resp. lower) sheet.

Next, suppose that for a $1$-cell, $e^1_\beta$, we are given a chain isomorphism
\[
f: (V(e^1_\beta), \widehat{d}_-) \rightarrow (V(e^1_\beta), \widehat{d}_+)
\]
where $\widehat{d}_-$ and $\widehat{d}_+$ are the boundary differentials associated to the intial and terminal vertices of $e^1_\beta$.  
 In addition, let $e^1_\beta \stackrel{j}{\rightarrow} e^2_\alpha$ be an appearance of $e^1_\beta$ along the boundary of $e^2_\alpha$.  (Technically, a lift of $e^1_\beta$ to the domain of the characteristic map of $e^2_\alpha$.)  
 We extend $f$ to a {\bf boundary morphism}
\[
\widehat{f}= \widehat{f}(e^1_\beta \stackrel{j}{\rightarrow} e^2_\alpha) : (V(e^2_\alpha), \widehat{d}_-) \rightarrow (V(e^2_\alpha), \widehat{d}_+)
\]
using the direct sum decomposition $V(e^2_\alpha)= i(V(e^1_\beta))\oplus V_{\mathit{cusp}}$ as
\[
\widehat{f} = f \oplus \mathit{id}.
\]

\subsubsection{Adjustments for swallowtail points}  
Suppose now that $e^0_{st}$ is an upward swallowtail point.   (The downward case is similar.)  Label adjacent cells as $e^1_S$, $e^1_T$, $e^1_{cr}$, $e^2_S$, and $e^2_T$ so that $e^2_S$ and $e^2_T$ contain the corners labeled $S$ and $T$; $e^1_{cr}$ contains the crossing locus; and $e^1_S$ and $e^1_T$ sit below the cusp edges that border the $S$ and $T$ corners.   See Figure \ref{fig:LocalST}.

\begin{figure}
\labellist
\small
\pinlabel $e^1_T$ [tl] at 82 62
\pinlabel $e^2_T$  at 66 76
\pinlabel $e^1_S$ [tr] at 12 62
\pinlabel $e^2_S$  at 32 76
\pinlabel $T$  at 54 34
\pinlabel $S$  at 44 34
\pinlabel $e^1_{\mathit{cr}}$ [l] at 50 56
\pinlabel $e^0_{\mathit{st}}$ [l] at 54 2
\endlabellist
\centerline{\includegraphics[scale=1]{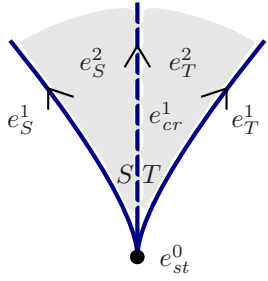}}

\quad

\caption{Notation for cells near a swallowtail point.}
\label{fig:LocalST}
\end{figure}

We make the following adjustments:

\begin{enumerate}
\item Given $(V(e^0_{st}), d)$, the boundary differentials for $V(e^2_{T})$, $V(e^1_{cr})$ and $V(e^2_{S})$ are defined as follows.  Above  $e^2_T$, the sheets of $L$ are totally ordered, so we write
\[
L(e^2_T)  = \{S_{1}, \ldots, S_{n}\} \quad \mbox{with $S_{i} \prec S_{i+1}$.}  
\]
Let $S_k, S_{k+1}, S_{k+2}\in L(e^2_T) $ be the sheets whose closures contain the swallowtail point, so that $S_{k+1}$ and $S_{k+2}$ meet at the crossing arc. First, we define $d_{k,k+1}:V(e^2_T) \rightarrow V(e^2_T)$ as if sheets $S_k$ and $S_{k+1}$ meet at a cusp above $e^0_{st}$, i.e. identify $V(e^0_{st})$ with the subspace spanned by $\{S_{1}, \ldots, \widehat{S_k}, \widehat{S_{k+1}}, S_{k+2}, S_{n}\} \subset L(e^2_T)$, and extend $d$ to $d_{k,k+1}$ via $d_{k,k+1} S_{k+1} = S_k$.  Then, define the
\begin{equation} \label{eq:dTdef}
\widehat{d}_T 
=  h_{k+1,k+2} d_{k,k+1} h_{k+1,k+2}
\end{equation}
where $h_{k+1,k+2}$ is the handleslide map, $h_{k+1,k+2}(S_l) = S_l + \delta_{l,k+2} S_{k+1}$.

The boundary differentials on $V(e^1_{cr})$ and $V(e^2_{S})$ are defined so that the bijections $L(e^2_S) \cong L(e^1_{cr}) \cong L(e^2_{T})$ (from identifying sheets whose closures intersect above $e^1_{cr}$) extend to isomorphisms of complexes.  Note that if sheets above $e^2_S$ are also labeled with descending $z$-coordinate, then the isomorphism $Q:V(e^2_T) \rightarrow V(e^2_S)$ interchanges $S_{k+1}$ and $S_{k+2}$.  Because of this,
 the boundary differential $\widehat{d}_S:V(e^2_S) \rightarrow V(e^2_S) $ would be
\begin{equation} \label{eq:dSdef}
\widehat{d}_S = Q \widehat{d}_T Q^{-1}= h_{k+2,k+1} d_{k,k+2} h_{k+2,k+1}
\end{equation}
with $d_{k,k+2}$ is formed as if $S_k$ and $S_{k+2}$ meet at a cusp above $e^0_{st}$.

Boundary differentials for $V(e^1_{T})$ and $V(e^1_{S})$ (and neighboring cells outside the swallowtail region) are defined using the bijection $L(e^1_S) \cong L(e^0_{st}) \cong L(e^1_{T})$.   

\item Suppose we have a chain isomorphism $f: (V(e^1_X), \widehat{d}_-) \rightarrow (V(e^1_X), \widehat{d}_+)$, for $X=S$ or $T$.  (Here, $\widehat{d}_-$ is the differential for the swallowtail point, since we have assumed in Convention \ref{conv:1} all $1$-cells are oriented away from the swallowtail point.)  We define the boundary morphism $\widehat{f}: (V(e^2_X), \widehat{d}_-) \rightarrow (V(e^2_X), \widehat{d}_+)$ via 
\[
\widehat{f} = (f \oplus \mathit{id}) \circ H_X
\]
where we decompose $V(e^2_X)$ in the usual way into $i(V(e^1_X)) \oplus V_{\mathit{cusp}}$, and $H_X$ is defined by
\begin{equation} \label{eq:HXdef}
H_S = \left(\prod_{\{i|\langle d_0S^0_{k}, S^0_i \rangle \ne 0\}} h_{i,k} \right) h_{k+1,k+2}, \quad \mbox{and} \quad H_T = h_{k+1,k+2}
\end{equation}
where $d_0$ denotes the differential on $V(e^0_{st})$.  (Note that sheets above $e^0_{st}$ are totally ordered, and the handleslide maps in the product all commute.)
\end{enumerate}

\begin{remark}
See Figure \ref{fig:Endpoints} and proof of Proposition \ref{prop:GFMorse2} for a Morse theoretic explanation of the maps $H_S$ and $H_T$.
\end{remark}

\begin{lemma}  \label{lem:XST} For $X=S$ or $T$, the boundary morphisms $\widehat{f}: (V(e^2_X), \widehat{d}_-) \rightarrow (V(e^2_X), \widehat{d}_+)$ are chain maps.
\end{lemma}

\begin{proof}
Note that $f \oplus \mathit{id}$ is a chain map from $(V(e^2_X), d_{k,k+1}) \rightarrow (V(e^2_X), \widehat{d}_+)$ (since the differentials respects the direct sum $i(V(e^1_x)) \oplus V_{\mathit{cusp}}$, and $f$ was a chain map).  
Thus, it suffices to check that $H_S$ and $H_T$ are chain maps from $(V(e^2_X), \widehat{d}_-)$ to $(V(e^2_X), d_{k,k+1})$. 

In the notation from (\ref{eq:DifferentialST}),  with respect to the basis $S_1, \ldots, S_n$ for $V(e^2_X)$, (sheets ordered with descending $z$-coordinate above $e^2_X$) the relevant linear maps have the following matrices:

\centerline{\begin{tabular}{c|c}   
Linear Map &  Matrix \\
\hline 
$d_{k,k+1}$ &  $\widehat{A}_{k,k+1}$ \\
For $X=S$,  & \\
$\widehat{d}_-$ &  $A_S$ \\
$H_S$ & $S = I + \widehat{A}_{k,k+1} E_{k+2,k} + E_{k+1,k+2}$ \\
For $X=T$,  & \\
$\widehat{d}_-$ &  $A_T$ \\
$H_T$ & $T = I + E_{k+1,k+2}$ 
\end{tabular}}

\noindent where the entries of the underlying $(n-2)\times(n-2)$ matrix $A$ are specialized as
\begin{equation}  \label{eq:dS}
a_{i,j} \mapsto \langle d_0S^0_{j}, S^0_{i} \rangle.
\end{equation}
[For the matrix for $H_S$, start with the definition of $H_S$ to compute
\begin{align*}
\left(\prod_{\{i|\langle d_0S^0_{k}, S^0_i \rangle \ne 0\}} (I+E_{i,k}) \right) (I+E_{k+1,k+2}) & = \left(\prod_{i<k}(I+a_{i,k}E_{i,k})\right) (I+E_{k+1,k+2}) \\
 & = I + \sum_{i<k}a_{i,k}E_{i,k} +E_{k+1,k+2} = S.]
\end{align*}  

Thus, we need to verify the matrix identities
\begin{equation}  \label{eq:Akk1}
\widehat{A}_{k,k+1}S= SA_S \quad \mbox{and} \quad \widehat{A}_{k,k+1} T = T A_T.
\end{equation}
In \cite[Lemma 3.4]{RuSu1}, the equations 
\[
\partial S = \widehat{A}_{k,k+1}S + SA_S \quad \mbox{and} \quad \partial T = \widehat{A}_{k,k+1} T + T A_T
\]
are established in the cellular DGA.  Since $\partial T = 0$, and $\partial S = (\widehat{A}_{k,k+1})^2 E_{k+2,k}$, the left hand sides vanish once $a_{i,j}$ is specialized as in (\ref{eq:dS}) (since then $\widehat{A}_{k,k+1}$ is the matrix of a differential).

\end{proof}

\subsection{Augmentations as Chain Homotopy Diagrams}

\begin{definition}  \label{def:CHD}
A {\bf Chain Homotopy Diagram} for $(L,\mathcal{E})$  is a triple $\mathcal{D} =(\{d_\alpha\}, \{f_\beta\}, \{K_\gamma\})$ consisting of
\begin{enumerate}
\item  For each $0$-cell, a differential, $d_\alpha$, making $(V(e^0_\alpha), L(e^0_\alpha), d_\alpha)$ into an ordered complex.
\item  For each $1$-cell, a chain map $f_\beta:(V(e^1_\beta), \widehat{d}_-) \rightarrow (V(e^1_\beta), \widehat{d}_+)$ such that $f_\beta -\mathit{id}$ is strictly upper triangular.  Here, $\widehat{d}_-$ and $\widehat{d}_+$ denote the boundary differentials associated to the $0$-cells at the initial and terminal endpoint of $e^1_\beta$. 
\item  For each $2$-cell, a strictly upper triangular chain homotopy $K_\gamma: (V(e^2_\gamma), \widehat{d}_{v_0}) \rightarrow (V(e^2_\gamma), \widehat{d}_{v_1})$ between the chain maps $\widehat{f_j}^{\eta_j} \circ \cdots \circ \widehat{f_1}^{\eta_1}$ and $\widehat{f_m}^{\eta_m}\circ \cdots \circ \widehat{f_{j+1}}^{\eta_{j+1}}$.  Here, $\widehat{d}_{v_0}$ and $\widehat{d}_{v_1}$ denote the boundary differentials for the vertices $v_0$ and $v_1$;  the $\widehat{f}_{i}$ with $ 1 \leq i \leq j$ (resp. with $j+1 \leq i \leq m$) are the boundary morphisms associated to the  edges of $e^2_\gamma$ as they appear in the counter-clockwise (resp. clockwise) path from $v_0$ to $v_1$ in the domain of a characteristic map for $e^2_\gamma$; and the exponents are $+1$ (resp. $-1$) when the orientation of the $1$-cell agrees (resp. disagrees) with the orientation of this path. 
\end{enumerate}
\end{definition}

Suppose $L$ is equipped with a Maslov potential $\mu$ so that the vector spaces $V(e^d_\alpha)$ are all graded by $\Z/m(L)$.  Given a divisor $\rho \, | \, m(L)$, we say that a CHD $\mathcal{D}$ 
 is {\bf $\rho$-graded} if the maps $d_\alpha$, $f_\beta$, and $K_\gamma$ all have respective degrees $+1$, $0$, and $-1$ mod $\rho$.

\begin{proposition}
\label{prop:AugmentationCHD}

For any $\rho\,|\, m(L)$, there is a bijection between $\rho$-graded augmentations of the Cellular DGA of $(L,\mathcal{E})$ and $\rho$-graded Chain Homotopy Diagrams for $(L,\mathcal{E})$.
\end{proposition}

\begin{proof}
First, consider triples of linear maps $(\{d_\alpha\}, \{f_\beta\}, \{K_\gamma\})$ with the only restriction being that each $d_\alpha$, $f_\beta-\mathit{id}$, and $K_\gamma$ is strictly upper triangular.  There is a bijection between such triples and the set of all algebra homomorphisms from the cellular DGA $\A$ to $\Z/2$ that arises from replacing a linear map with its matrix with respect 
to $L(e^d_\alpha)$:
\[ (\{d_\alpha\}, \{f_\beta\}, \{K_\gamma\})  \mapsto  \left(\epsilon:\mathcal{A} \rightarrow \Z/2 \right), \]
\[
 \epsilon(a^\alpha_{p,q}) = \langle d_{\alpha}S^\alpha_q, S^\alpha_p\rangle, \quad \epsilon(b^\beta_{p,q}) = \langle f_{\beta}S^\beta_q, S^\beta_p\rangle, \quad \epsilon(c^\gamma_{p,q}) = \langle K_{\gamma}S^\gamma_q, S^\gamma_p\rangle.
\] 
[This is a bijection because all matrix coefficients of the $(\{d_\alpha\}, \{f_\beta-id\}, \{K_\gamma\})$ corresponding to pairs $S_p$ and $S_q$ for which there is no corresponding generator of $\mathcal{A}$ are forced to be $0$ by the strictly upper triangular condition, eg. the generator $a^\alpha_{p,q}$ exists if and only if  $S^\alpha_p \prec S^\alpha_q$.]

The above correspondence restricts to a bijection between CHDs and augmentations since the requirements on the maps $d_\alpha$, $f_\beta$, and $K_\gamma$ from the definition of CHD are equivalent to the matrix equations arising from applying $\epsilon \circ d =0$ to the corresponding $A$, $B$, and $C$ matrices.  In more detail, we have:
\begin{enumerate}
\item For $A = (a^\alpha_{p,q})$, 
\[
\epsilon \circ \partial (A) = 0   \quad \Leftrightarrow \quad \left[\epsilon(A)\right]^2 = 0 \quad \Leftrightarrow \quad (\partial_\alpha)^2=0, 
\]
i.e. $\epsilon \circ \partial (A) = 0$ if and only if $(V(e^0_\alpha), d_\alpha)$ is a chain complex.
\item For $B= (b^\beta_{p,q})$, 
\[
\epsilon \circ \partial (B) = 0  \quad \Leftrightarrow \quad \epsilon(A_+)(I+\epsilon(B)) = (I+ \e(B)) \e(A_-) \quad \Leftrightarrow \quad \widehat{d}_+ \circ f_\beta = f_\beta \circ \widehat{d}_-,
\]
i.e. $\epsilon \circ \partial (B) = 0$ if and only if $f_\beta: (V(e^1_\beta), \widehat{d}_-) \rightarrow (V(e^1_\beta), \widehat{d}_+)$ is a chain map.  [Note that $(I+\epsilon(B))$ is the matrix of $f_\beta$.  It is also important to observe that the $\epsilon(A_{\pm})$ are the matrices for the boundary differentials $\widehat{d}_\pm$. This is readily verified from comparing the {\it boundary matrices} used in defining $\partial B$ with the {\it boundary differentials} associated to $V(e^1_\beta)$.  In particular, (i) the $2\times 2$ blocks $\left[ \begin{array}{cc} 0 & 1 \\ 0 & 0 \end{array} \right]$ inserted when forming $A_\pm$ reflect the definition of $\widehat{d}_\pm$ on the subspace $V_{\mathit{cusp}} \subset V(e^1_\beta)$, and (ii) as already observed in Lemma \ref{lem:XST} when $e^1_\beta$ is the crossing $1$-cell at a swallowtail point, $A_- = A_S$ (or $A_T$ depending on whether the chosen total ordering of $L(e^1_\beta)$ used to form $B$ agrees with the ordering above the $S$ or $T$ $2$-cell) is the matrix of the boundary differential $\widehat{d}_-$.]  
\item  For $C= (c^\gamma_{p,q})$, considering first the case that $e^2_\gamma$ does not border swallowtail points, 
\begin{align*}
& \epsilon \circ \partial (C) = 0 \\
 \Leftrightarrow \quad & \e(A_{v_1})\e(C) + \e(C) \e(A_{v_0}) = (I+\e(B_{j}))^{\eta_j} \cdots (I+ \e(B_1))^{\eta_1} + (I+\e(B_{m}))^{\eta_m} \cdots (I+\e(B_{j+1}))^{\eta_{j+1}} \\
\Leftrightarrow \quad & 
\widehat{d}_{v_1}K_\gamma + K_\gamma  \widehat{d}_{v_0} = \widehat{f_j}^{\eta_j} \circ \cdots \circ \widehat{f_1}^{\eta_1} - \widehat{f_m}^{\eta_m}\circ \cdots \circ \widehat{f_{j+1}}^{\eta_{j+1}},
\end{align*}
i.e. $\epsilon \circ \partial (C) = 0$ if and only if  $K_\gamma: (V(e^2_\gamma), \widehat{d}_{v_0}) \rightarrow (V(e^2_\gamma), \widehat{d}_{v_1})$ is a chain homotopy between the chain maps $\widehat{f_j}^{\eta_j} \circ \cdots \circ \widehat{f_1}^{\eta_1}$ and $\widehat{f_m}^{\eta_m}\circ \cdots \circ \widehat{f_{j+1}}^{\eta_{j+1}}$.  [We used that since the $B_i$ are nilpotent, 
\[
\e(I+B_i)^{-1} = \e(I+B_i + B_i^2 + \cdots) = I + \e(B_i) + \e(B_i)^2+\cdots = (I+\e(B_i))^{-1}.
\]
Again, it is important to verify that $\epsilon(A_{v_i})$ (resp. $I+ \epsilon(B_i)$) is the matrix of the boundary differential associated to $v_i$ 
(resp. boundary morphism for the corresponding $f_i$).  The case of boundary differentials is as before, while the   $\left[ \begin{array}{cc} 0 & 0 \\ 0 & 0 \end{array} \right]$ inserted into $B_i$ is consistent with $\widehat{f}_i$ acting as the identity on the component $V_{cusp} \subset V(e^2_\gamma)$.]

In the case that $e^2_\gamma$ contains the $S$ or $T$ corner at a swallowtail point, the definition of the $\widehat{f}_i$ for the cusp edge bordering the corner acquires a factor of $H_S$ or $H_T$, while an $S$ or $T$ matrix is inserted at the corresponding part of the product in the definition of $\partial C$. As observed in Lemma \ref{lem:XST}, $S$ and $T$ are the respective matrices of $H_S$ and $H_T$, so it follows that $\e \circ \partial(C) =0$ is still equivalent to $K_\gamma$ being a chain homotopy of the required form.   
\end{enumerate} 
\end{proof}

\section{Morse Complex 2-families}
\label{sec:Morse2CHD}

 In this section, we introduce Morse complex 2-families (abbr. MC2Fs) which are detailed combinatorial approximations of generating families.  In Section \ref{sec:Comb}, using an MC2F we produce combinatorial continuation maps associated to paths in the base surface, again in analogy with Morse theory.  
  Finally, in Proposition \ref{prop:GFMorse2} we show that pairing a generating family $F$ with an appropriate family of gradient-like vector fields produces an  MC2F, and we observe how properties of $F$ are reflected in the associated continuation maps.

\subsection{Definition of MC2Fs}

Let $L \subset J^1M$  with Maslov potential $\mu$ have generic front and base projections.  We write 
\[
\Sigma = \Pi_B( \Sigma_{\mathit{cusp}} \cup \Sigma_{\mathit{st}} \cup \Sigma_{\mathit{cr}})
\]
 for the base projection of the singular set of $L$ (cusps, swallowtail points, and crossing arcs).
Let $R_\nu \subset M \setminus \Sigma$ be a region, i.e. an open connected subset.
Following earlier definitions, we let $L(R_\nu)$ denote the set of sheets of $L$ above $R_\nu$, i.e. components of $\pi_B^{-1}(R_\nu) \cap L$.  
Sheets in $L(R_\nu)$ are totally ordered by descending $z$-coordinate, so we always index sheets as $L(R_\nu) = \{S^\nu_1, S^\nu_2, \ldots, S^\nu_n\}$ with $z(S_i) > z(S_{i+1})$ pointwise.  The $\Z/2$-vector space spanned by $L(R_\nu)$ is denoted $V(R_\nu)$, and
 is assigned a $\Z/m$-grading via the Maslov potential.

\begin{dfn}
\label{def:CM2F}  Let $\rho \, | \, m(L)$.  A $\rho$-graded {\bf  Morse complex 2-family} (abbrv. {\bf MC2F}), $\mathcal{C}$, for $L$ is a triple $\mathcal{C}= (\{d_\nu\},H,H_{-1})$ which consists of the following data:

\begin{enumerate}
\item A {\bf super-handleslide set}, $H_{-1}$, which is a finite set of points in $M \setminus \Sigma$.
 Each point $x \in H_{-1}$ is assigned upper and lower 
lifts,
$u_x, l_x \in L$ satisfying 
\[
z(u_x) > z(l_x), \quad \mbox{and} \quad \mu(u_x) -\mu(l_x) =-1 \,(\mbox{mod $\rho$}).  
\]
\item A {\bf handleslide set} which is an immersed compact $1$-manifold $H: X \rightarrow M$ where $X = \bigsqcup_i X_i$ with each $X_i = S^1$ or $[0,1]$.  
When restricted to the interior of $X$, $H$ is transverse to (the strata of) $\Sigma$; is disjoint from $H_{-1}$; and the only self-intersections are transverse double points in $M \setminus \Sigma$.   Moreover, $H$ is equipped with continuous upper and lower endpoint lifts, $u,  l: X \rightarrow L$ 
 satisfying
\[
\quad z(u)> z(l), \quad \quad \mu(u) -\mu(l) =0 \, (\mbox{mod $\rho$}).
\]
\item Set
\[
\Sigma_\mathcal{C} = \Sigma \cup H(X) \cup H_{-1}.
\]For each connected component $R_\nu \subset M \setminus \Sigma_\mathcal{C}$,
 the vector space $V(R_\nu)$ is assigned a differential $d_\nu$ making
 $(V(R_\nu), L(R_\nu), d_\nu)$ into a $\rho$-graded ordered complex, i.e.  $d_\nu$ is strictly upper triangular and
 \[
 \mbox{deg}(d_\nu) =+1 \, (\mbox{mod $\rho$}).
 \] 

\end{enumerate}

The data $(\{d_\nu\},H,H_{-1})$ is subject to the two Axioms \ref{ax:endpoints} and \ref{ax:iso}.

\end{dfn}

Before stating Axioms \ref{ax:endpoints} and \ref{ax:iso} we introduce some terminology.  
When considering the handleslide set of $\mathcal{C}$ locally in $M\setminus \Sigma$, a handleslide arc whose upper (resp. lower) lift is $S_{i}$ (resp. $S_j$) is called an  {\bf $(i,j)$-handleslide arc}.   Note that the indices $i$ and $j$ are not globally well-defined for a given component of $H$, since they may change when the image of $H$ crosses $\Sigma$.  The phrase {\bf $(i,j)$-super-handleslide point} has a similar meaning.

\begin{figure}

\quad

\quad

\labellist
\small
\pinlabel (1) [r] at -16 64
\pinlabel $h_{k,j}$ [b] at 16 116
\pinlabel $h_{i,j}$ [b] at 48 124
\pinlabel $h_{i,k}$ [b] at 128 116
\pinlabel $S_i$ [l] at 586 112
\pinlabel $S_k$ [l] at 586 64
\pinlabel $S_j$ [l] at 586 16
\endlabellist
\centerline{\includegraphics[scale=.6]{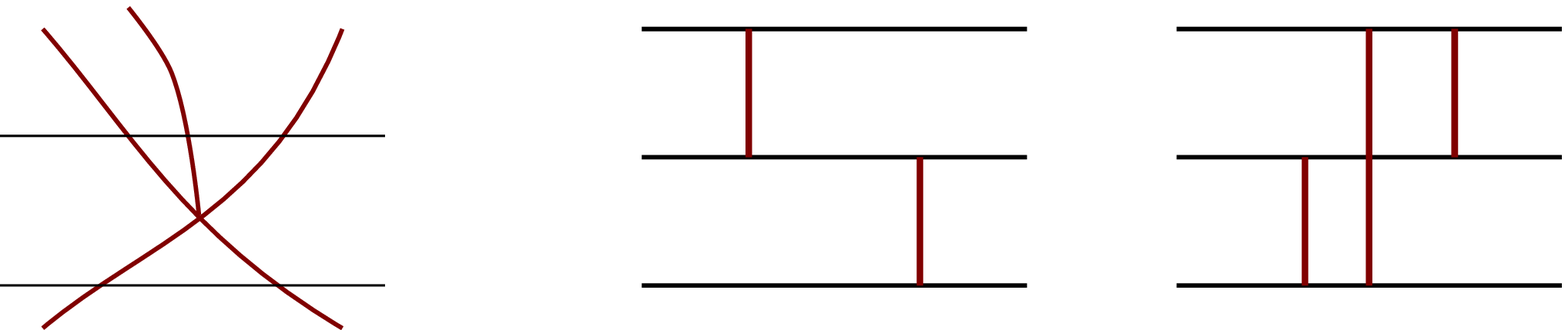}} 

\quad

\quad

\quad

\labellist
\small
\pinlabel (2) [r] at -16 64
\pinlabel $S_{u}$ [l] at 586 68
\pinlabel $S_{l}$ [l] at 586 34
\endlabellist

\centerline{\includegraphics[scale=.6]{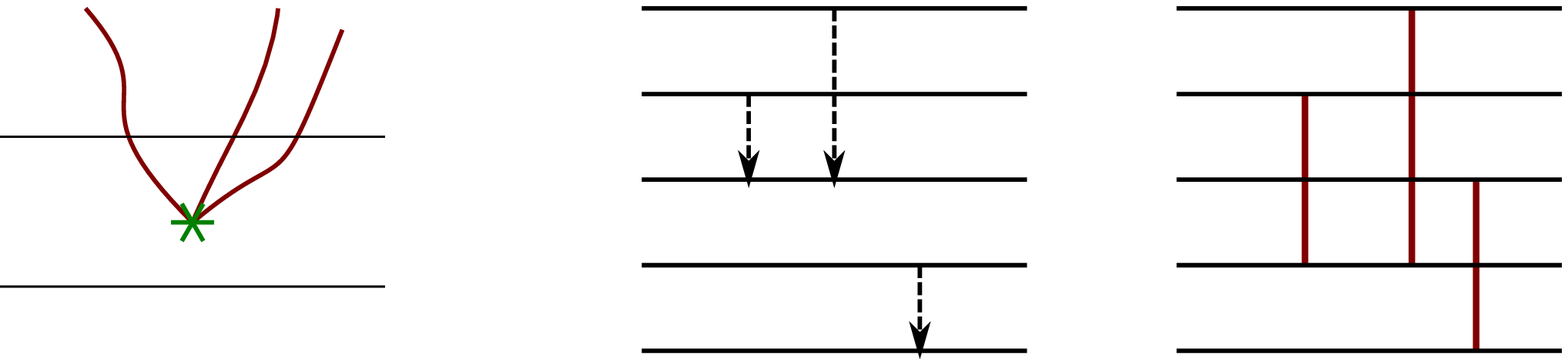}}

\quad

\quad

\quad

\labellist
\small
\pinlabel (3) [r] at -16 152
\pinlabel $x_1$ [t] at 32 -2
\pinlabel $x_2$ [r] at -2 32
\pinlabel $x_1$ [t] at 272 -2
\pinlabel $z$ [r] at 238 32
\pinlabel $H_S$ [t] at 486 36
\pinlabel $H_T$ [t] at 542 36
\pinlabel $S$ [b] at 40 220
\pinlabel $T$ [b] at 104 220
\endlabellist

\centerline{\includegraphics[scale=.6]{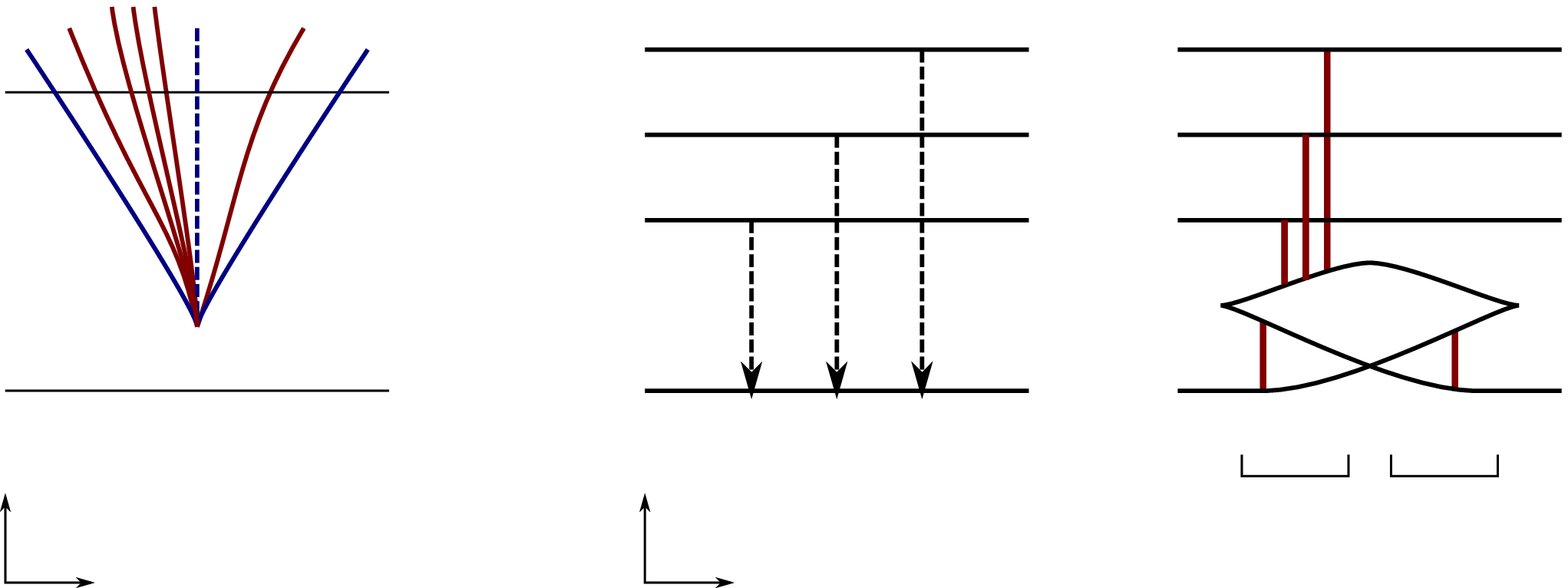}}

\caption{The three types of endpoints for handleslide arcs in $H$ allowed by Axiom \ref{ax:endpoints}.  The left column depicts the base projection (to $M$) of $H$ (in red), $H_{-1}$ (a green star) and the singular set, $\Sigma$, (in blue).  
The center and right column depict two slices of the front projection; a dotted black arrow  from $S_i$ to $S_j$ indicates that $\langle \partial S_j, S_i \rangle = 1$.  
The three types of endpoints allowed are:  (1) at double points of $H$,  (2) at super-handleslide points, and (3) at swallowtail points.}   
\label{fig:Endpoints}
\end{figure}

\begin{ax}
\label{ax:endpoints}
Endpoints of handleslide arcs (for components of $H$ with $X_i= [0,1]$) are as follows; see also Figure \ref{fig:Endpoints}.

\begin{enumerate}
\item Let $x\in M\setminus \Sigma$ be a double point of $H$ where for some $i<m<j$ an $(i,m)$-handleslide arc intersects an $(m,j)$-handleslide arc.  Then, a unique  $(i,j)$-handleslide arc has a unique endpoint at $x$.



\item Suppose  $p \in H_{-1}$ is a $(u,l)$-super-handleslide point, and
let $d_\nu$ be the differential associated to any region of $M \setminus \Sigma_\mathcal{C}$ adjacent to $p$.
Then, for any $i < u< l < j$, at  $p$ there are $\langle  d_\nu S_{u},  S_i \rangle$ endpoints of $(i,l)$-handleslide arcs; and  
$\langle  d_\nu S_j,  S_{l}\rangle$ endpoints of $(u,j)$-handleslide arcs.

\item Suppose  $p \in M$ is an upward swallowtail point such that outside (resp. inside) the swallowtail region $L$ has $n-2$ (resp. $n$) sheets, and such that sheet $S_k$ (resp. sheets $S_k$, $S_{k+1}$, and $S_{k+2}$) contains the swallowtail point in their closure.  

Denote by $d_{0}$ the differential associated to the $n-2$ sheeted region of $M\setminus \Sigma_\mathcal{C}$ near $p$.
Then, at $p$ there are $\langle d_0 S_{k}, S_{i}\rangle$ endpoints of $(i,k)$-handleslide arcs locally contained within the swallowtail region
 as well as $2$ additional $(k+1,k+2)$-handleslide arcs, one on each side of the crossing locus near $p$.

The downward swallowtail case is similar, but vertically reflected.  
\end{enumerate}

\end{ax}

\begin{ax}

\label{ax:iso}

When two regions $R_0$ and $R_1$ share a border along an arc, $A \subset \Sigma_\mathcal{C}$, the complexes $(V(R_0), d_0)$ and $(V(R_1), d_1)$ are related as follows:

\begin{enumerate}

\item Suppose $A$ belongs to an $(i,j)$-handleslide arc.  
We require that the handleslide map 
\[
h_{i,j}: (V(R_0), d_0) \stackrel{\cong}{\rightarrow} (V(R_1), d_1)
\]
is a chain isomorphism.

\item Suppose $A$ belongs to the crossing locus.  We have a bijection $L(R_0)\cong L(R_1)$ by identifying sheets whose closures (in $L$) intersect above $A$.  We require that the induced isomorphism $V(R_0) \cong V(R_1)$ is an isomorphism of complexes.  

Equivalently, 
label sheets above $R_0$ and $R_1$ with descending $z$-coordinate as $S^0_1, \ldots, S^0_n$ and $S^1_1,\ldots, S^1_n$.  If sheets $S^{i}_k$ and $S^i_{k+1}$ meet at the crossing arc above $A$, we require that the map
\[
Q: (V(R_0), d_0) \stackrel{\cong}{\rightarrow} (V(R_1), d_1),  \quad Q(S^0_{i}) = S^1_{\tau(i)}
\]
is an isomorphism where $\tau = (k \,\, k+1)$ denotes the transposition.

\item Suppose $A$ belongs to the cusp locus.  We require that the complexes are related 
as in the boundary differential construction of Section \ref{sec:BoundaryDiff}.

In more detail, suppose that above $A$ the sheets $S^1_{k}$ and $S^1_{k+1}$ meet at a cusp edge.  Include $V(R_0)$ into $V(R_1)$ via 
\[
S^{0}_i  \mapsto \left\{\begin{array}{lr} S^1_{i}, & i<k \\ S^1_{i+2}, & i \geq k, \end{array} \right.
\]
and write $V_{\mathit{cusp}}= \mbox{Span}_{\Z/2}\{S^1_{k}, S^1_{k+1}\}$.  We require that, with respect to the direct sum decomposition $V(R_1)=V(R_0)\oplus V_{\mathit{cusp}}$, the differential is $d_1= d_0 \oplus d_{\mathit{cusp}}$ where $d_{\mathit{cusp}}S_{k+1} = S_{k}$.

\end{enumerate}

\end{ax}

We record some observations about the definition.

\begin{observation}
\label{ob:CM2F}
\begin{enumerate}

\item   For Axiom \ref{ax:endpoints} (2) about the appearance of $H$ near a super-handleslide $p\in H_{-1}$ it suffices to check the condition for a single choice of adjacent region at $p$.  It then follows from Axiom \ref{ax:iso} (1) that the condition will hold for all adjacent regions, since the differentials associated to different regions bordering $p$ are related by a sequence of handleslide maps that do not change the matrix coefficients $\langle d S_u, S_i \rangle$ and $\langle d S_j, S_l\rangle$ with $i<u<l<j$. 

\item If sheets $S_k$ and $S_{k+1}$ cross along at least one boundary arc of a region $R_\nu$ then $\langle d_vS_{k+1}, S_{k}\rangle = 0 $.  [This follows from Axiom \ref{ax:iso} (2).  Otherwise, the differential in the neighboring region would not be upper triangular.]

\item If sheets $S_k$ and $S_{k+1}$ meet at a cusp along at least one boundary arc of a region $R_\nu$ then $\langle dS_{k+1}, S_{k}\rangle  =1$.  [Use Axiom \ref{ax:iso} (3).]

\item  An $(i,j)$-handleslide arc 
cannot intersect a crossing locus
involving sheets $S_i$ and $S_j$, and cannot cross a cusp edge involving $S_i$ or $S_j$.  [This is because the lifts satisfy the inequality $z(u) >z(l)$, and cannot be continuously extended past a cusp point.]

\item Given a swallowtail point $p$ and a differential $d_0$ for the outside of the swallowtail region, once handleslide arcs are placed near $p$ as required in Axiom \ref{ax:endpoints} (3), at least locally, there is always a unique way to assign differentials $\{d_\nu\}$ to the regions within the swallowtail region so that Axiom \ref{ax:iso} holds.  See Proposition \ref{prop:Extend2}. 

\end{enumerate}

\end{observation}

\begin{example}
A $0$-graded MC2F for a Legendrian in $J^1 \R^2$ is pictured in Figure \ref{fig:MC2FEx}.  The two green *'s are $(2,3)$-super handleslide points.  The lower red arc is a $(1,3)$-handleslide arc.  The upper red arc is a $(2,4)$-handleslide arc (resp. $(3,4)$-handleslide arc) when it is outside (resp. inside) the crossing circle. The differentials $d_\nu$ are indicated by the dotted arrows.
\end{example}

\begin{figure}
\labellist
\small
\pinlabel $x_1$ [t] at 0 32
\pinlabel $x_2$ [t] at 64 54
\pinlabel $z$ [b] at 18 110
\endlabellist
\centerline{\includegraphics[scale=.5]{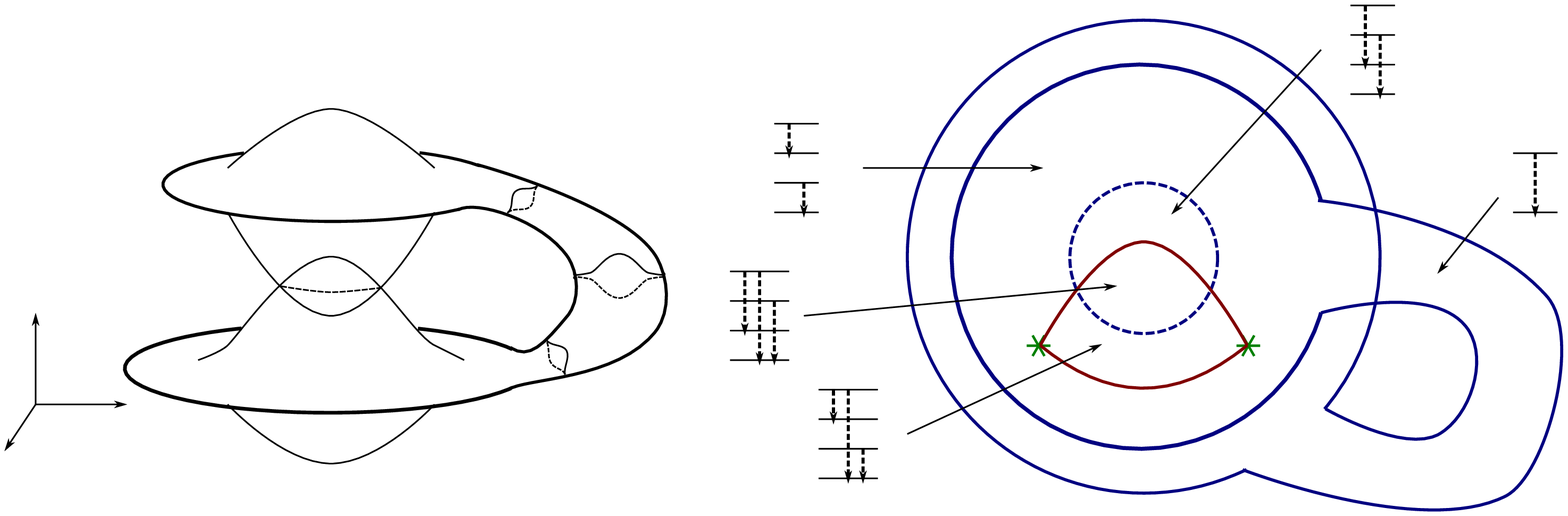}}

\caption{An example of an MC2F (right) for a Legendrian $L \subset J^1\R^2$ with pictured front diagram (left).}
\label{fig:MC2FEx}
\end{figure}

\begin{remark}  
\label{rmk:GromovCoMorse}
\begin{enumerate}
\item The definition of an MC2F is based on the generic bifurcations of Morse complexes in $2$-parameter families of functions; see Proposition \ref{prop:GFMorse2}.  That the  differentials $d_\nu$ have degree $+1$ (mod $\rho$) corresponds to working with Morse cohomology complexes rather than homology. Here, the grading is given by the Morse index, but the differential counts positive gradient trajectories rather than negative trajectories.  

\item The reader familiar with the Gromov compactness/gluing proof of $d^2=0$ in Morse or Floer theory can interpret  Axiom \ref{ax:endpoints} (1) and (2) as gluing various configurations of broken trajectories to produce boundaries of the moduli space of handleslide trajectories. 
\end{enumerate}
\end{remark}


\subsection{Combinatorial continuation maps} \label{sec:Comb}

Suppose that $\mathcal{C}= (\{d_\nu\},H,H_{-1})$ is a MC2F for $L \subset J^1M$.  In the following, using $\mathcal{C}$ we associate continuation maps to paths in $M$.  For paths that are disjoint from the singular set of $L$ the continuation maps have properties at the chain level that will be important for constructing a CHD from an MC2F.  

\medskip

Let $\sigma:[0,1] \rightarrow M$ be a smooth path that is 
transverse to the strata of $\Sigma_\mathcal{C}$. 
Suppose $\sigma(i)$ lies in the component $R_i \subset M \setminus \Sigma_\mathcal{C}$ for $i=0,1$.  
 We define the {\bf continuation map}
\begin{equation}
\label{eq:ContinuationMap}
f(\sigma):(V(R_0),d_0) \rightarrow (V(R_1), d_1)
\end{equation}
to be the composition 
\[
f(\sigma) =  f_m\cdots f_1
\] with the maps $f_1, \ldots, f_m$ associated to those $0 < s_1 < \ldots < s_m <1$ where $\sigma(s_l)$ intersects $\Sigma_\mathcal{C}$ 
as follows:
\begin{enumerate}
\item When $\sigma(s_l)$ intersects an $(i,j)$-{\it handleslide},  
\[
f_l = h_{i,j}.
\]
\item When $\sigma(s_l)$ intersects a {\it crossing}, $f_l$ is the map $Q$ from Axiom \ref{ax:iso} (2).
\item When $\sigma(s_l)$ intersects a {\it cusp}, notate the regions bordering the cusp edge as  $R'$ and $R''$ so that the two cusp sheets exist above $R''$ and not above $R'$.  Write $V(R'') = V(R') \oplus V_{\mathit{cusp}}$.  If $\sigma$ passes from $R'$ to $R''$ as $s$ increases, then $f_l:V(R') \rightarrow V(R'')$ is the inclusion.  If $\sigma$ passes from $R''$ to $R'$, then $f_l:V(R'') \rightarrow V(R')$ is the projection.  
\end{enumerate}


\begin{proposition}  \label{prop:ContinuationMap}
Let $\sigma,\tau: [0,1] \rightarrow H$ be paths transverse to $\Sigma_{\mathcal{C}}$.
\begin{enumerate}
\item The continuation map $f(\sigma)$ is a quasi-isomorphism.
\item If $\sigma(1) = \tau(0)$, then 
\[
f(\sigma * \tau) = f(\tau) \circ f(\sigma).
\]

\item If $\sigma$ and $\tau$ are path homotopic (i.e. homotopic relative endpoints) in $M$, then $f(\sigma), f(\tau) : (V(R_0),d_0)\rightarrow (V(R_1),d_1)$ are chain homotopic. 
\end{enumerate}

If $\sigma$ and $\tau$ are disjoint from the singular set of $L$, i.e. disjoint from crossing and cusp arcs, then: 

\begin{enumerate}
\item[(4)] The matrix of $f(\sigma) - \mathit{id}$ is strictly upper-triangular.

\item[(5)] The inverse path $\sigma^{-1}(s) = \sigma(1-s)$ has
\[
f(\sigma^{-1}) = (f(\sigma))^{-1}.
\]

\item[(6)]  If $\sigma$ and $\tau$ are path homotopic via a homotopy whose image is also disjoint from crossings and cusps, then there is a strictly upper-triangular homotopy operator, $K:V(R_0) \rightarrow V(R_1)$ between $f(\tau)$ and $f(\sigma)$,
\[
f(\sigma)-f(\tau) = d_1 K + K d_0.  
\]
If the image of the homotopy is also disjoint from super-handleslide points, then $f(\sigma) = f(\tau)$.
\end{enumerate}
When $\mathcal{C}$ is $\rho$-graded, all of the above continuation maps (resp. homotopy operators) have degree $0$ (resp. $-1$) mod $\rho$.
\end{proposition}

\begin{proof}
This is based on one standard approach to continuation maps in Morse Theory, as in \cite{Laud}. 

(1) follows from Axiom \ref{ax:iso} which shows that each individual factor $f_l:(V(R_{l-1}), d_{l-1}) \rightarrow (V(R_l), d_l)$ is a quasi-isomorphism where $R_{l-1}$ (resp. $R_{l}$) are the regions containing $\sigma(s)$ as $s \rightarrow s_{l}^-$ (resp. as $s \rightarrow s_{l}^+$).
 
(2) is obvious from the definition.  

(4) and (5) follow from the definition since $h_{i,j}^{-1} = h_{i,j}$, and the matrix of each $h_{i,j}$ 
is upper triangular with $1$'s on the diagonal.  

To prove (6), we consider a homotopy from $\sigma$ to $\tau$, $I:[0,1] \times[0,1] \rightarrow M$, $I(s,t) = \sigma_t(s)$, with $\sigma_t(i) = \sigma(i) = \tau(i)$, for $i=0,1$, such that the image of $I$ is disjoint from all crossing and cusp arcs.  By taking $I$ sufficiently generic, 
we can assume $I^{-1}(H)$ is an immersed $1$-manifold whose non-embedded points are as in the definition of MC2F, i.e. the interior of $I^{-1}(H)$ has at worst  transverse double points, and all endpoints of $I^{-1}(H)$ in the interior of $[0,1]\times[0,1]$ are as in Axiom \ref{ax:endpoints} (1) and (2).    
Moreover, we can assume the projection to the $t$ direction is a Morse function $\pi_t:I^{-1}(H)  \rightarrow \R$, and all critical points of $\pi_t,$  double points of $I^{-1}(H)$, and super-handleslide points occur at different values of $t$.   We subdivide $0 = t_0 < t_1 < \ldots < t_N=1$ so that each interval $[t_i,t_{i+1}]$ contains only one such $t$-value that is located in the interior of the interval.  See Figure \ref{fig:ContinuationMap}.  To complete the proof, we check that $f(\sigma_{t_i}) \sim f(\sigma_{t_{i+1}})$.  

\medskip

\noindent{\bf Case 1:} $(t_i,t_{i+1})$ contains a {\it critical point} of $\pi_t|_{I^{-1}(H)}$.  Then, the products that define $f(\sigma_{t_i})$ and $f(\sigma_{t_{i+1}})$ agree except for a consecutive pair of handleslides maps $h_{i,j} h_{i,j}$ that appears in only one of the two.  Since $h_{i,j}^{-1} = h_{i,j}$ we get $f(\sigma_{t_i}) = f(\sigma_{t_{i+1}})$.   

\medskip

\noindent{\bf Case 2:} $(t_i,t_{i+1})$ contains an interior {\it double point} of $I^{-1}(H)$.

For any $u_1>l_1$ and $u_2>l_2$, a straightforward computation gives the relations for handleslide maps,
\begin{equation}  \label{eq:Relations}
\begin{array}{cr} h_{u_1,l_1}  h_{u_2,l_2} = h_{u_2,l_2} h_{u_1,l_1},  & \mbox{when $l_1 \neq u_2 $ and $l_2 \neq u_1$,} \\
h_{u_1,l_1}h_{u_2,l_2} = h_{u_2,l_2} h_{u_1,l_1} h_{u_1,l_2}, & \mbox{when $l_1=u_2$.}
\end{array}
\end{equation}

Let $u_1>l_1$ and $u_2>l_2$ denote the indices of the upper and lower lifts of the two interior points of $I^{-1}(H)$ that intersect. 
If $l_1 \neq u_2 $ and $l_2 \neq u_1$, then $f(\sigma_{t_i})$ and $f(\sigma_{t_{i+1}})$ differ by the transposition of a pair of consecutive factors: that is  $h_{u_1,l_1} h_{u_2,l_2}$ is interchanged with $h_{u_2,l_2} h_{u_1,l_1}$.  The first formula from (\ref{eq:Relations}) shows that $f(\sigma_{t_i}) = f(\sigma_{t_{i+1}})$.

Supposing that $l_1=u_2$, Axiom \ref{ax:endpoints} (1) applies to show that the products defining $f(\sigma_{t_i})$ and $f(\sigma_{t_{i+1}})$ are related as in the second equation of (\ref{eq:Relations}) with the caveat that the $h_{u_1,l_2}$ may appear in some other location, including on the left hand side.  Since $h_{u_1,l_2}$ is self-inverse and commutes with $h_{u_2,l_2}$ and $h_{u_1,l_1}$, the equality $f(\sigma_{t_i})=f(\sigma_{t_{i+1}})$ follows.

\medskip

\noindent{\bf Case 3.}  $(t_i,t_{i+1})$ contains a $(u,l)$-{\it super handleslide point}, $p$.

We can factor 
\[
f(\sigma_{t_i})= g f_a h, \quad \mbox{and} \quad f(\sigma_{t_{i+1}}) = g f_b h,
\]
\[
h:(V(R_0), d_0) \rightarrow (V(R'), d'), \quad f_a, f_b: (V(R'),d') \rightarrow (V(R''),d''), \quad g:(V(R''),d'') \rightarrow (V(R_1),d_1),
\]
where $f_a$ and $f_b$ correspond to the segments of $\sigma_{t_i}$ and $\sigma_{t_{i+1}}$ that contain the intersections of these paths with the collection of handleslides with endpoints at $p$, as in Axiom \ref{ax:endpoints} (2).  See Figure \ref{fig:ContinuationMap}. Since 
 any two of the handleslides with endpoints at $p$ give handleslide maps $h_{i_1,j_1}$ and $h_{i_2,j_2}$  with $j_1 \neq i_2$, (because $i_2\leq u< l \leq j_1$),  the matrix of $f_a-f_b$ is
\[
\sum_{i<u } \langle d''S_{u}, S_i \rangle E_{i,l} + \sum_{l<j} \langle d' S_j, S_{l}\rangle E_{u,j}.
\] 
(As in Observation \ref{ob:CM2F}(1), the coefficients $\langle d S_{u}, S_i \rangle$ and $\langle d S_j, S_{l}\rangle$ are the same when $d$ is the differential from any of the regions that border $p$, including $d'$ and $d''.$) 
Taking $K$ to have matrix $E_{u,l}$ it follows that
\[
f_a-f_b = d'' K +K d',
\]
so that
\[
f(\sigma_{t_i}) - f(\sigma_{t_{i+1}})= g f_a h - g f_b h = d_1( gKh) + (gKh) d_0.
\]
Note that since $g$ and $h$ (resp. $K$) are upper triangular (resp. strictly upper triangular), it follows that the homotopy operator $gKh$ is strictly upper triangular.

\medskip

With Case 1-3 established, we note that a homotopy operator $\tilde{K}$ between $f(\sigma)$ and $f(\tau)$ is the sum of the homotopy operators $K_i$ between each $f(\sigma_{t_i})$ and $f(\sigma_{t_{i+1}})$.   Thus, it follows that $\tilde{K}$ is indeed upper triangular, and is $0$ if the image of $I$ is disjoint from super-handleslide points. 

\medskip

Finally, to establish (3), the previous argument is extended to allow the possibility that the image of the homotopy $I$ intersects crossings and cusps.  Assuming $I$ generic, this leads to several new codimension $2$ strata of $I^{-1}(\Sigma_{\mathcal{C}})$ to be considered in producing the chain homotopy $f(\sigma_{t_i}) \sim f(\sigma_{t_{i+1}})$.  The list includes:
\begin{itemize}
\item[(a)] Local max/min's of $\pi_t$ restricted to a crossing or cusp arc.  
\item[(b)] Transverse crossings of two crossing, cusp, or handleslide arcs.  In the case of the intersection of two crossing and/or cusp arcs, we may assume that two disjoint pairs of sheets are involved.
\item[(c)]  The generic codimension $2$ singularities of front projections as in Figure \ref{fig:Codim2Both}:  Triple Points, Cusp-Sheet Intersections, and Swallowtail Points.  
\end{itemize}
We leave this straightforward, but somewhat lengthy case-by-case check mostly to the reader, commenting here on a few interesting points.  

Note that in fact $f(\sigma_{t_i}) = f(\sigma_{t_{i+1}})$ in all cases except some local maxima/minima of cusp arcs.  In the case of a local minimum, an identity map factor in $f(\sigma_{t_i})$ is replaced with either $j\circ p$ or $p \circ j$ where 
\[
V(R') \stackrel{j}{\rightarrow} V(R'') = V(R') \oplus V_{\mathit{cusp}},  \quad \mbox{and} \quad V(R'') = V(R') \oplus V_{\mathit{cusp}} \stackrel{p}{\rightarrow} V(R')
\]
are the inclusion and projection.  One has 
\[
p \circ j = \mathit{id}, \quad \mbox{and} \quad j \circ p - \mathit{id} = d_{R''} K + K d_{R''}
\]
where $K(S_a) = S_b$ for the cusp sheets $S_a$ and $S_b$ (with $S_a$ above $S_b$) and $K(S_i) = 0$ for $i \neq a$.  

We examine also the case of an (upward) swallowtail point.  The tangency to the cusp edge at the swallowtail can be assumed to be non-vertical, and we consider the case where the swallowtail sheets exist above $\sigma(t_{i+1})$ but not $\sigma(t_i)$.   Assuming the swallowtail sheets are $S_{k}, S_{k+1}, S_{k+2}$, so that the sheets meeting at cusp edges are labeled $S_k$ and $S_{k+1}$,   the continuation map $f(\sigma_{t_{i+1}})$ is obtained from $f(\sigma_{t_i})$ via inserting the product
\[
p H_S Q H_T j,
\]
where $H_S$, $Q$, and $H_T$ have matrices  
\[
H_S = I+E_{k+1,k+2}+ \sum_{i<k} a_{i,k} E_{i,k},  \quad Q=Q_{k+1,k+2}, \quad  H_T = I+E_{k+1,k+2}
\]
with $Q_{k+1,k+2}$ the permutation matrix for $(k+1 \, k+2)$ and $a_{i,k}\in \Z/2$.  (All of the handleslides specified in Axiom \ref{ax:endpoints} (3) with lower lift on $S_k$ are collected into the $H_S$ matrix; this is possible since each $h_{i,k}$ commutes with $Q$.)
Thus, for $i\neq k$ we compute
\[
(p H_S Q H_T j)(S_i) = \left\{ \begin{array}{lr} (p H_S Q H_T)(S_i),  & i<k, \\ (p H_S Q H_T)(S_{i+2}), & i>k \end{array} \right. =  
\left\{ \begin{array}{lr} p(S_i),  & i<k, \\ p(S_{i+2}), & i>k \end{array} \right. = S_i;
\]
\[
(p H_SQH_T j)(S_{k}) = (pH_SQH_T)(S_{k+2})= pH_SQ(S_{k+1}+S_{k+2}) = p H_S(S_{k+2}+S_{k+1}) = p(S_{k+2})= S_k.
\]

\end{proof}

\begin{figure}
\labellist
\small
\pinlabel $h$ [r] at 318 32
\pinlabel $h$ [l] at 402 32
\pinlabel $g$ [l] at 402 192
\pinlabel $g$ [r] at 318 192
\pinlabel $f_a$ [r] at 318 110
\pinlabel $f_b$ [l] at 402 110
\endlabellist
\centerline{\includegraphics[scale=.6]{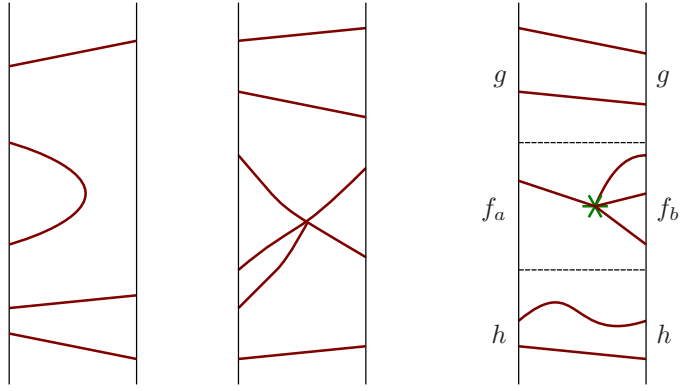}}

\caption{The handleslide set $I^{-1}(H)$ for $t \in [t_{i},t_{i+1}]$ as considered in Case 1, 2, and 3 from the proof of Proposition \ref{prop:ContinuationMap}.}
\label{fig:ContinuationMap}
\end{figure}

Let $x_0$ be a basepoint, belonging to a region $R_0\subset M \setminus \Sigma_\mathcal{C}$.  

\begin{corollary} \label{cor:Cx0}

\begin{enumerate}
\item The homology $H(\mathcal{C}_{x_0}) := H(V(R_0),d_0)$ is independent of the choice of $x_0$ and $R_0$.

\item The continuation maps induce a well defined anti-homomorphism
\[
\Phi_{\mathcal{C},x_0}: \pi_1(S, x_0) \rightarrow \mathit{GL}(H(\mathcal{C}_{x_0})),     \quad [\sigma] \mapsto H(f(\sigma)).
\]
\end{enumerate}
\end{corollary}

\begin{proof}
Follows from Proposition \ref{prop:ContinuationMap} (1)-(3).
\end{proof}

We refer to $H(\mathcal{C}_{x_0})$ as the {\bf fiber homology} of $\mathcal{C}$ at $x_0$, and $\Phi_{\mathcal{C},x_0}$ as the {\bf monodromy representation}.

\begin{remark}  \label{rem:localsystem}
Although we have only defined $H(\mathcal{C}_{x_0}, \Phi_\mathcal{C})$ for $x_0 \in M \setminus \Sigma_\mathcal{C}$, it is standard that a representation of the fundamental group at any point $x_0 \in M$ of a connected space extends to a local system of vector spaces, well-defined up to isomorphism. In this way, the representation $\Phi_{\mathcal{C},x_0}$ is defined up to isomorphism for arbitrary $x_0 \in M$.
\end{remark}

\subsection{Generating families and MC2Fs}

\begin{prop}
\label{prop:GFMorse2}
If the Legendrian $L$ has a tame at infinity 
  generating family $F$, then it has a $0$-graded Morse complex 2-family, $\mathcal{C}$. Moreover:
\begin{enumerate}
\item if $F$ is linear at infinity, then we can take $\mathcal{C}$ to have vanishing fiber homology, $H(\mathcal{C}_{x_0}) = \{0\};$
\item if the domain of $F$ is a trivial bundle over $M$, then we can take $\mathcal{C}$ to have trivial monodromy representation.  
\end{enumerate}
\end{prop}


\begin{proof}
Let $F:E \rightarrow \R$ be a generating family for $L \subset J^1M$ with fiber $N.$
In an open set $U \subset M$ above which $E$ is trivialized, we can consider $F$ as a 2-parameter family of smooth functions,  
$\{f_m:N \rightarrow \R\}_{m \in U}$.  
As discussed in \cite[p.22-23]{HatcherWagoner}, after generic small perturbation there is a stratification
$M = \mathcal{F}_0 \cup \mathcal{F}_1 \cup \mathcal{F}_2$ given by the critical points and values of the $f_m$.
 In the codimension-0 $\mathcal{F}_0$ stratum, all critical points are non-degenerate and critical values are distinct.
 The codimension-1 $\mathcal{F}_1$ stratum is the union of parameter values with a single birth-death or two non-degenerate points with a common critical value.
 The codimension-2 $\mathcal{F}_2$ stratum has six types of singularities: a unique swallowtail point and five various configurations of  transverse intersections of the codimension-1 strata.
The set $\mathcal{F}_1 \cup \mathcal{F}_2$ is the base projection of the singular set of $L$, $\Sigma = \Pi_B(\Sigma_{cusp} \cup \Sigma_{cr} \cup \Sigma_{st}),$ made of the cusp loci, crossing loci, their various intersections and the swallowtail points. 

A sheet of $\Pi_F(L)$ that lies above $U \subset M$ corresponds to a family of non-degenerate critical points $q_m$ of $f_m$ for $m \in U$ whose Morse indices $i_{mo}(q_m)$ are locally-constant. 
Seen this way, the Morse index of critical points provides a  $\Z$-valued Maslov potential on $L$.   This implies $m(L) = 0$ and gives the grading on vector spaces for which the $0$-graded requirements in Definition \ref{def:CM2F} are satisfied.
Similarly, the locally well-defined relative Morse index of two such families of critical points equals the difference in Maslov potentials of the two corresponding sheets.

We review several properties of the stable and unstable manifolds of critical points that can be arranged following \cite{HatcherWagoner}.  
  In order to produce the simplest behavior near cusps and swallowtail points, it is useful to have the property that all non-degenerate critical points have $1 \leq i_{mo}(p) \leq n-1$ where $n = \dim N$.  This condition holds after stabilizing $F$ via the quadratic form  $Q(\mu_1,\mu_2) = \mu_1^2-\mu_2^2$. 
When forming ascending and descending manifolds in the non-compact, but tame at infinity setting we use gradient-like vector fields that agree outside of a compact set  with the Euclidean gradient of the linear or quadratic function that $F$ is equal to at infinity.  

Following \cite{HatcherWagoner}, 
there exists a 2-family, $\{g_m,V_m\}_{m\in M}$, of metrics $g_m$ and gradient-like vector fields  $V_m$ (on the fibers of $E$) for the functions $f_m,$ 
such that the following hold:
\begin{enumerate}
\item For all $m \in \mathcal{F}_0$ and $p_m \in \mbox{Crit}(f_m),$  the stable and unstable manifolds $W^s(p_m)$ and $W^u(p_m)$ vary smoothly with $(m, p_m),$ i.e. the fiber-wise stable and unstable manifolds of sheets of $L$ are smooth manifolds.     
\item For all $m$ near the points in $\mathcal{F}_1$ with a pair of ``near birth-death" points  $p_m^+$  and $p_m^-$ with $i_{mo}(p_m^+) - i_{mo}(p_m^-)=1,$  $W^u(p_m^+)$ and $W^s(p_m^-)$ intersect transversely at an intermediary level-set in one point.  
\item For all $m \in M$ and $p_m, q_m  \in \mbox{Crit}(f_m)$ with locally well-defined relative Morse index, $i_{mo}(p_m) - i_{mo}(q_m),$ equal to $1,0,-1$ the unions (over $m$) of $W^s(p_m)$ and  $W^u(q_m)$ are in general position.
\item
Outside of arbitrary small disk neighborhoods, $N(e^0_{\mathit{st}})$, of the swallowtail points, all the birth/death points are  {\it independent}.
An independent birth/death is one in which the stable (resp. unstable) manifolds of the newly-born pair of points do not intersect the unstable (resp. stable) manifolds of the other critical points.
\end{enumerate}
These items follow from Theorem 3.1 on p. 42, p.52-53, p.62-63, and Chapter IV, Section 2, Part (C) of \cite{HatcherWagoner}.

We now translate these items into the language of Definition \ref{def:CM2F} to construct a Morse complex 2-family.
Consider a pair of families of non-degenerate critical points $p_m, q_m$.
If $i_{mo}(p_m) - i_{mo}(q_m) = -1,$ then the set of $m \in M$ such that $W^u(p_m) \cap W^s(q_m)  \ne \emptyset$ is a set of points which we use to define $H_{-1}.$ 
If $i_{mo}(p_m) - i_{mo}(q_m) = 0,$ then the set of $m \in M$ such that $W^u(p_m) \cap W^s(q_m) \ne \emptyset$ is a collection of curves in general position which we use to define $H$ outside of $\cup N(e^0_{\mathit{st}})$.  Both $H_{-1}$ and $H$ have natural upper and lower lifts to $L$ specified by the image of the critical points $p_m,q_m \in \Sigma_F$ under the diffeomorphism $i_F:\Sigma_F \rightarrow L$.  (Notation as in Section \ref{sec:GeneratingFamilies}.)  As in  Chapter IV, Section 2, Part (C), page 147 \cite{HatcherWagoner}, the intersection with $\partial N(e^0_{\mathit{st}})$ of handleslide arcs with lifts on the swallowtail sheets is as specified by Axiom \ref{ax:endpoints} (3) where the differential $d_0$ is the differential from the Morse complex of the $f_m$ outside the swallowtail region.    We complete the definition of $H$ by connecting these handeslide endpoints to the swallowtail point.  As a technical point,  the number of $(i,k)$-handleslide arcs only agrees with $\langle d_0S_k, S_i\rangle$ mod $2$;  if necessary, we can connect any extra endpoints in pairs.

We now assign differentials $d_\nu$ to components $R_\nu$ of $\mathcal{F}_0 \setminus (H_{-1} \cup H)= M \setminus \Sigma_\mathcal{C}.$ 
First, consider regions outside of $\cup N(e^0_{\mathit{st}})$.  We can assume that for generic $m \in R_\nu,$ the gradient-like vector field $V_m$ of $f_m$ is Morse-Smale. 
We can then define $d_\nu$ as the Morse co-differential, which counts positive flows of $V_m$ between critical points of relative Morse index $1$. See Remark \ref{rmk:GromovCoMorse}. 
 This differential is independent of the choice of $m \in R_\nu$, since any other such $m' \in R_\nu$ can be connected to $m$ by a path in $R_\nu$ along which the Morse-Smale condition holds except at finitely many points where two flowlines between the same pair of critical points of the $f_m$ appear or disappear.
 This does not change $d_\nu$.  Finally, note that there is a unique way to assign differentials in $\cup N(e^0_{\mathit{st}})$ so that Axiom \ref{ax:iso} holds.  [If necessary, see Propositions \ref{prop:Extend1} or \ref{prop:Extend2} below.]
 
We now verify that Axioms \ref{ax:endpoints} and \ref{ax:iso} follow from known Cerf theory, subject to the convention-reversing modification in Remark \ref{rmk:GromovCoMorse}.   That all endpoints for handleslide arcs are as in Axiom \ref{ax:endpoints} is established over the course of Chapter IV of \cite{HatcherWagoner} which needs a complete
treatment of 2-parameter families of functions and gradient-like vector fields for its invariance
proof of the (Morse) $K$-theoretic $Wh_2$ pseudo-isotopy invariant.
Endpoints as in Axiom  \ref{ax:endpoints}(1) are discussed in Chapter II Section 1, page 89 \cite{HatcherWagoner}.
Endpoints as in Axiom \ref{ax:endpoints}(2) appear in the ``Exchange Relation," see Chapter IV, Section 2, Part (A), page 131 \cite{HatcherWagoner}.
Near swallowtail points,  Axiom \ref{ax:endpoints}(3) follows from the ``Dovetail Relation," see Chapter IV, Section 2, Part (C), page 147 \cite{HatcherWagoner}.

 Axiom \ref{ax:iso}(1) is immediate, since when passing the crossing locus thru a point $m$ that is disjoint from handleslides, swallowtail, or cusp points, the Morse complex remains unchanged, except for the ordering of generators by critical value.  
 Axiom \ref{ax:iso}(2) is a well-known result \cite[Section 7]{Milnor}.
 Axiom \ref{ax:iso}(3) follows from items (2) and (4) of the list of properties for the stable and unstable manifolds of the critical points (see earlier in this proof).
 
Thus, we have produced an MC2F, $\mathcal{C}$, from a tame at infinity generating family.  It remains to establish (1) and (2) from the statement of the proposition. 
 
For (1), observe that the fiber homology $H(\mathcal{C}_{x_0})$ is the cohomology of the Morse complex of $f_{x_0}$ (the restriction of $F$ to the fiber above $x_0$).  Assuming $F$ linear at infinity, $f_{x_0}$ has the form 
\[
f_{x_0}: E'_{x_0} \times \R^k \rightarrow \R
\]
 where $E'_{x_0}$ is the (compact) fiber of $E'$ above $x_0$, and agrees with a non-zero linear function $l: \R^{k} \rightarrow \R$ outside of a compact set.  We can split $\R^k \cong \ker l \oplus \R$, and by 
 compactifying the $\ker l$ factor, we can extend $f_{x_0}$ to a smooth function 
 \[
 f_{x_0}: E'_{x_0} \times S^{k-1} \times \R \rightarrow \R
 \] that (i) is proper and (ii) agrees with the projection to the $\R$ factor outside of a compact set.  This extension does not change the Morse complex of $f_{x_0}$, and in this setting the Morse complex computes the relative cohomology of $(f \leq T, f \leq -T)$ where $T >>0$; see for instance \cite{Milnor}.  Since 
\[
(f\leq T, f\leq -T) =  \left(E'_{x_0} \times S^{k-1} \times (-\infty, T], E'_{x_0} \times S^{k-1} \times (-\infty, -T]\right),
\]
 it follows that $H(\mathcal{C}_{x_0}) = \{0\}$.


To prove (2), assume $E \rightarrow M$ is the trivial bundle $ M \times N.$   (By the tame at infinity assumption, $N = N' \times \R^k$ with $N'$ compact.) 
Let $\sigma$ be a loop in $M$, generic with respect to the base projection of the singular set.  The induced generating family on $S^1$ (with trivial bundle domain  $S^1 \times N$), call it $F_{S^1}$, extends to a tame at infinity generating family on $D^2$ (with domain $D^2 \times N$).  [This is because the subset of $C^\infty(N ,\R)$ consisting of those functions  agreeing with a fixed linear or quadratic function on $\R^k$ at infinity is contractible.]  
Taking the extension of $F_{S^1}$ to $D^2 \times N$ to be sufficiently generic, the transversality condition in the definition of generating families will hold and the front projection of the resulting Legendrian on $J^1D^2$ will be generic.  This Legendrian is equipped with an MC2F, $\mathcal{C}',$ such that the continuation map for $\mathcal{C}' $ associated to the $S^1$ boundary loop of $D^2$ agrees with the continuation map for $\sigma.$  By Proposition \ref{prop:ContinuationMap} (3), this continuation map induces the identity map on homology (since it is chain homotopic to the continuation map for a constant loop).

 \end{proof}

\section{From MC2F to CHD}
\label{sec:Proof}

In this section, we show how to construct a CHD, and hence an augmentation, from an MC2F.   
A key technical point in associating a CHD to an MC2F is to allow continuation maps to be associated to the edges of a compatible polygonal decomposition for $L$.  This is not immediate from Section \ref{sec:Comb} since edges may be contained in the singular set of $L$, but is accomplished by shifting $0$-cells and $1$-cells off of the singular set.  See Figure \ref{fig:GFtoCHD1} for a summary.

\subsection{Continuation maps associated to edges of a compatible cell decomposition}  \label{sec:Shifts}

\begin{figure}
\labellist
\small
\endlabellist
\centerline{\includegraphics[scale=.6]{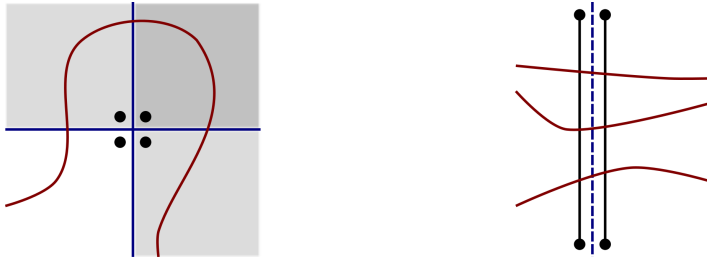}}

\caption{Given an MC2F for $L$, differentials $d(e^0_\alpha \rightarrow e^0_\alpha):V(e^0_\alpha) \rightarrow V(e^0_\alpha)$ are defined by shifting vertices $e^0_\alpha$ into bordering $2$-cells (left). Continuation maps $f(e^1_\beta\rightarrow e^1_\beta)$ are assigned to $1$-cells by a similar shift (right).    The choice of shift is non-unique and well-definedness is verified in Propositions \ref{prop:dmaps} and \ref{prop:fmaps}.}
\label{fig:GFtoCHD1}
\end{figure}

Let $\mathcal{E}$ be a compatible polygonal decomposition for $L$ satisfying Convention \ref{conv:1}.

\begin{definition}  \label{def:nice}

A MC2F is {\bf nice} with respect to $\mathcal{E}$ if
\begin{enumerate}
\item  The handleslide sets are transverse to the $1$-skeleton of $\mathcal{E}$ except at swallowtail points which may be endpoints of handleslide arcs (as in Axiom \ref{ax:endpoints} (3)).
\item  In a neighborhood of each upward swallowtail point, the $(i,k)$-handle slide arcs 
(as in Axiom \ref{ax:endpoints} (3)) are contained in the corner labeled $S$, while both the $S$ and $T$ corners contain a $(k+1,k+2)$-handleslide arc.  A similar condition is imposed at downward swallowtail points.
\end{enumerate}

\end{definition}

Let $\mathcal{C} = (\{d_\nu\},H,H_{-1})$ be an MC2F  that is nice with respect to $\mathcal{E}$.  Recall that the $d_\nu$ are differentials on $V(R_\nu)$ where $\{R_\nu\}$ is the set of connected components  of $M \setminus \Sigma_\mathcal{C}$ (with $\Sigma_\mathcal{C}$ the union of the 
handleslide sets of $\mathcal{C}$ and the singular set of $L$).

Using $\mathcal{C}$, we now associate to each appearance of a vertex $e^0_\beta$ in the closure of another cell, $e^0_\beta \stackrel{j}{\rightarrow} e^d_\alpha$, (notation as in Section \ref{sec:BoundaryDiff})  
 a differential
\[
d(e^0_\beta \rightarrow e^d_\alpha): V(e^d_\alpha) \rightarrow V(e^d_\alpha).   
\]

\begin{itemize}

\item  Assuming $e^0_\beta$ is {\it not a swallowtail point}:  Choose a component $R_\nu$ 
 whose closure contains a neighborhood of $e^0_\beta$ in  $\overline{e^d_\alpha}$.  The sheets $L(e^d_\alpha)$ are identified with a subset of $L(R_\nu)$ in the usual way, 
 so that
 \begin{equation} \label{eq:LRnu}
 L(R_\nu) = L(e^d_\alpha) \sqcup L_{\mathit{cusp}}
 \end{equation}
 with the sheets in $L_{\mathit{cusp}}$ meeting in pairs at cusp edges above $e^d_\alpha$.
Axiom \ref{ax:iso} (2) and (3) imply that in the resulting direct sum $V(R_\nu) = V(e^d_\alpha) \oplus V_{\mathit{cusp}}$ the $V(e^d_\alpha)$ component is a sub-complex of $(V(R_\nu), d_\nu)$.  Thus, we can define
\[
d(e^0_\beta \rightarrow e^d_\alpha) = d_\nu|_{V(e^d_\alpha)}.
\]


\item Assuming $e^0_\beta$ is {\it a swallowtail point}:   When $e^d_\alpha$ is one of $e^2_S$, $e^2_T$ or $e^1_{cr}$ we identify $L(e^d_\alpha)$ with $L(R_s)$, $L(R_t)$, or $L(R_t)$ respectively where $R_t$ (resp. $R_s$) is the region that borders the crossing locus on the side labeled $T$ (resp. $S$).  Take the corresponding $d_t$ or $d_s$ for $d(e^0_\beta \rightarrow e^d_\alpha)$.  For any other $e^d_\alpha$, the sheets $L(e^d_\alpha)$ are identified bijectively with $L(R_0)$ where $R_0$ is the region that borders $e^0_\beta$ from outside the swallowtail region; the resulting isomorphism $V(e^d_\alpha) \cong V(R_0)$ allows us to put $d(e^0_\beta \rightarrow e^d_\alpha) = d_0$. See Figure \ref{fig:RsRt}.
\end{itemize}

\begin{figure}

\quad

\labellist
\small
\pinlabel $R_s$ [b] at 68 90
\pinlabel $R_t$ [b] at 92 90
\pinlabel $R_0$ [t] at 44 8
\pinlabel $\sigma_S$ [b] at 240 122
\pinlabel $\sigma_T$ [b] at 400 122
\endlabellist
\centerline{\includegraphics[scale=.6]{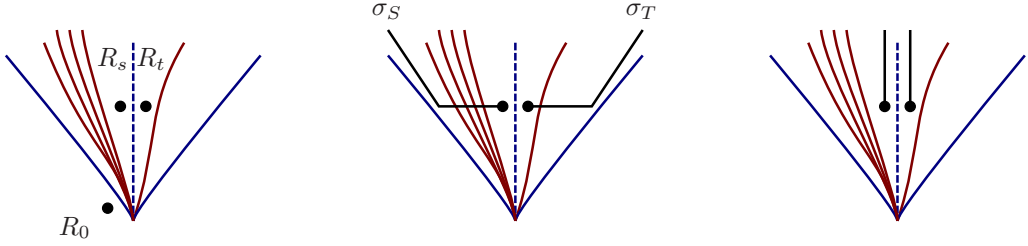}}

\caption{(left) The regions $R_s$, $R_t$, and $R_0$ at a swallowtail point used to define the $d(e^0_{st}\rightarrow e^d_\alpha)$.  The paths that define  $f(e^1_S \rightarrow e^2_S)$ and $f(e^1_T \rightarrow e^2_T)$ (center) and $f(e^1_{cr}\rightarrow e^d_\alpha)$ (right).}
\label{fig:RsRt}
\end{figure}

\begin{proposition} \label{prop:dmaps}
\begin{enumerate}
\item The differentials $d(e^0_\beta \rightarrow e^d_\alpha)$ are well defined.

\item For any $e^0_\beta \rightarrow e^d_\alpha$, $d(e^0_\beta \rightarrow e^d_\alpha)$ is the boundary differential associated to $d_\beta:= d(e^0_\beta \rightarrow e^0_\beta)$  (as in Section \ref{sec:BoundaryDiff}).
\end{enumerate}
\end{proposition}

\begin{proof}
Well-definedness is only in question in the non-swallowtail case.  Suppose that $R_\nu$ and $R_\mu$ are two regions that border the cell $e^d_\alpha$ at the vertex $e^0_\beta$.  (For $d=1$ there could be $2$ such regions,  for $d=0$ there may be many.  See Figure \ref{fig:GFtoCHD1} for a concrete example.)  We can get from $R_\nu$ to $R_\mu$ by passing through a sequence of $1$-cells with a common endpoint at $e_\beta^0$.  Thus, we can assume without loss of generality that $R_\nu$ and $R_\mu$ share such a $1$-cell in their boundary.  Moreover, if that $1$-cell is a cusp edge we may assume the two cusp sheets exist above $R_\mu$ but not above $R_\nu$. 



The splitting from (\ref{eq:LRnu}) defines an inclusion $i_\nu: V(e^d_\alpha) \rightarrow V(R_\nu)$ and projection $p_\nu: V(R_\nu) \rightarrow  V(e^d_\alpha)$, and analogous maps $i_\mu$ and $p_\mu$ are defined for $R_\mu$.   We need to show that $D_\nu= D_\mu$ where  
\[
D_\nu = p_\nu \circ d_\nu \circ i_\nu \quad \mbox{and} \quad D_\mu = p_\mu \circ d_\mu \circ i_\mu.
\]
The Axiom \ref{ax:iso} (2) or (3) (depending if the $1$-cell where $R_\nu$ and $R_\mu$ meet is a crossing or a cusp) provides a chain map $h: (V(R_\nu), d_\nu) \rightarrow (V(R_\mu), d_\mu)$.  It is clear from the definitions that $h \circ i_\nu = i_\mu$, and  $p_\nu = p_\mu \circ h$, so the equality $D_\nu = D_\mu$ follows in a routine manner.
%



To check (2) in the non-swallowtail case, we may assume that the same region $R_\nu$ is used in defining $d_\beta= d(e^0_\beta \rightarrow e^0_\beta)$ and $d(e^0_\beta \rightarrow e^d_\alpha)$.  The sheets of  $L(e^d_\alpha)$ not identified with sheets of $L(e^0_\beta)$ occur in pairs that meet at a cusp above $e^0_\beta$.  From (2) and (3) of Axiom \ref{ax:iso}, it follows that  $d(e^0_\beta \rightarrow e^d_\alpha)$ takes $S_b \mapsto S_a$ for each such pair of cusping sheets (with $S_a$ the upper of the two sheets)  and agrees with $d_\beta$ on the span of $L(e^0_\beta) \subset L(e^d_\alpha)$.  Thus, $d(e^0_\beta \rightarrow e^d_\alpha)$ is indeed related to $d_\beta$ precisely as in the boundary differential construction of Section \ref{sec:BoundaryDiff}.

In the swallowtail case, $d_{st}=d(e^0_{st} \rightarrow e^0_{st})$ is the differential $d_0$ from the component $R_0$ outside the swallowtail region, and this is the same as $d(e^0_{st} \rightarrow e^d_\alpha)$ and the boundary  differential for all neighboring $e^d_\alpha$ except for $e^2_S$, $e^2_T$, and $e^1_{cr}$.  
  In Section \ref{sec:BoundaryDiff}, the associated boundary differential for $e^0_{st} \rightarrow e^2_T$ is defined as $\widehat{d}_T = h_{k+1,k+2}d_{k,k+1} h_{k+1,k+2}$ where $d_{k,k+1} = d_0 \oplus d_{\mathit{cusp}}$ using the isomorphism $V(e^2_T) = V(e^0_{st})\oplus V_{\mathit{cusp}}$ where the splitting arises from identifying $L(e^0_{st})$ with the subset $\{S_1,\ldots, \widehat{S_{k}},\widehat{S_{k+1}}, \ldots, S_n \} \subset L(e^2_T)$, and $d_{\mathit{cusp}}S_{k+1} = S_k$.  (Subscripts indicate ordering above $e^2_T$.)  To see that this $\widehat{d}_T$ agrees with $d(e^0_{st} \rightarrow e^2_T) = d_t$,  travel from the region $R_0$ to $R_t$ by passing first through the $e^1_T$ cusp edge and then across the $h_{k+1,k+2}$ handleslide arc that appears in the $T$ half of the swallowtail region; according to Axiom \ref{ax:iso} (3) and (1) the differential from the MC2F will change first from $d_0$  to $d_{k,k+1}$ and then to $h_{k+1,k+2}d_{k,k+1} h_{k+1,k+2}$ when we arrive at $R_t$; thus, $d_t = \widehat{d}_T$.  Next, apply Axiom \ref{ax:iso} (2) and the definition of $\widehat{d}_S$ from (\ref{eq:dSdef}) to see that
  \[
  d(e^0_{st} \rightarrow d^2_S) = d_s = Q d_tQ^{-1} = Q \widehat{d}_T Q^{-1} = \widehat{d}_S. 
  \]
  Finally, note that for  $e^1_{cr}$ the boundary differential and $d(e^0_{st} \rightarrow e^1_{cr})$ are defined to agree with $d_t$ and $\widehat{d}_T$ respectively.    
\end{proof}

Suppose that the $1$-cell $e^1_\beta$ has initial and terminal vertices $e^0_-$ and $e^0_+$.  For each inclusion $e^1_\beta \rightarrow \overline{e^d_\alpha}$ as an edge, we associate a morphism 
\[
f(e^1_\beta \rightarrow e^d_\alpha): (V(e^d_\alpha), d(e^0_- \rightarrow e^d_\alpha)) \rightarrow (V(e^d_\alpha), d(e^0_+ \rightarrow e^d_\alpha)).   
\]
In the case when $e^1_\beta = e^d_\alpha$, we refer to $f_\beta:= f(e^1_\beta \rightarrow e^1_\beta)$ as the {\bf continuation map} for the edge $e^1_\beta$.  

\begin{itemize}

\item Assuming $e^1_\beta$ {\it has no endpoints at swallowtails}:  Choose a neighboring $2$-cell $e^2_\gamma$ containing $e^d_\alpha$ in its closure.  (When $e^d_\alpha =e^1_{\beta}$, there are two choices; when $d=2$, $e^2_\gamma = e^d_\alpha$.)  Shift $e^1_\beta$ slightly to a path $\sigma$ contained in the interior of a collar neighborhood $e^1_\beta \times [0, \epsilon) \subset \overline{e^2_\gamma}$ that is disjoint from $H_{-1}$ and such that $e^0_{\pm} \times[0,\epsilon)$ is disjoint from $H$.  Let $R_-$ and $R_+$ denote the components that contain the shifts of $e^0_-$ and $e^0_+$.  
The continuation map 
\[
f(\sigma): (V(R_-), d_-) \rightarrow (V(R_+), d_+)
\]
 is well-defined 
by Proposition \ref{prop:ContinuationMap} (6).
 As usual, we can split $L(e^2_{\gamma}) = L(e^d_\alpha) \sqcup L_{\mathit{cusp}}$.  We can assume $\sigma$ does not intersect handleslide arcs from $H$ with endpoint lifts on sheets of $L_{\mathit{cusp}}$ (as in Observation \ref{ob:CM2F} (4)  these arcs are not allowed to reach the cusp edge).  Then, $f(\sigma)$ respects the decomposition $V(R_-) =V(R_+) = V(e^2_\gamma) =  V(e^d_\alpha) \oplus V_{\mathit{cusp}}$ and we define  $f(e^1_\beta \rightarrow e^2_\alpha)$ as the component
\begin{equation} \label{eq:e2gamma}
f(\sigma) = f(e^1_\beta \rightarrow e^d_\alpha) \oplus id : V(e^d_\alpha) \oplus V_{\mathit{cusp}} \rightarrow V(e^d_\alpha) \oplus V_{\mathit{cusp}}.
\end{equation}



\item Assuming $e^1_\beta$ {\it has an endpoint at a swallowtail}, $e^0_{st}$:  In view of Convention \ref{conv:1}, the endpoint at $e^0_{st}$ must be the initial point of $e^1_S$, $e^1_T$, and $e^1_{cr}$.  In the case of $e^1_{cr}$, the $f(e^1_{cr}\rightarrow e^d_\alpha)$ are defined as above.  For $e^1_S$, define
$f(e^1_S \rightarrow e^2_S) = f(\sigma_S)$ 
for a path $\sigma_S$ that starts in $R_s$ near the swallow tail point, runs perpendicularly across the handleslide arcs in the $S$ corner of the swallowtail region, and then runs parallel to $e^1_S$ (remaining on the side of $e^1_S$ where the cusp sheets exist).  For other $e^d_\alpha$, define $f(e^1_S \rightarrow e^d_\alpha)$ to be a continuation map for a path that is a shift of $e^1_S$ to the outside of the swallowtail region.

Define the $f(e^1_T \rightarrow e^d_\alpha)$ similarly.  See Figure \ref{fig:RsRt}.

\end{itemize}

\begin{proposition}\label{prop:fmaps}
\begin{enumerate}
\item The morphisms $f(e^1_\beta \rightarrow e^d_\alpha)$ are well defined.

\item For any $e^1_\beta \rightarrow e^d_\alpha$, $f(e^1_\beta \rightarrow e^d_\alpha)$ is the boundary map associated to $f_\beta$ (as in Section \ref{sec:BoundaryDiff}).
\end{enumerate}
\end{proposition}

\begin{proof}

We only need to verify well-definedness when $e^d_\alpha = e^1_\beta$.  Then, there are two competing shifts, $\sigma_a$ and $\sigma_b$, of $e^1_\beta$ into the two neighboring cells $e^2_a$ and $e^2_b$.  Since $H$ is transverse to $e^1_\beta$, assuming $\sigma_a$ and $\sigma_b$ are sufficiently close to $e^1_\beta$ there will be a bijection between the sequence of handleslide arcs appearing along the paths $\sigma_a$ and $\sigma_b$; specifically, the bijection identifies the endpoints of the components of the intersection of $H$ with $e^1_\beta \times [-\epsilon,\epsilon]$.   Moreover, above $\sigma_a$  and  $\sigma_b$ the endpoint lifts of these handleslides belong to the subsets $i_a(L(e^1_\beta)) \subset L(e^2_a)$   and  $i_b(L(e^1_\beta)) \subset L(e^2_b)$, and agree in $L(e^1_\beta)$.  Thus, the $V(e^1_\beta)$ component of the continuation maps $f_a$ and $f_b$ agree, as required.

For (2), we need to show that for $e^1_\beta \rightarrow e^2_\gamma$, the map $f(e^1_\beta \rightarrow e^2_\gamma)$ is the boundary morphism for $f(e^1_\beta \rightarrow e^1_\beta)$.  In the non-swallowtail case or in the case of a swallowtail with $e^1_\beta = e^1_{cr}$, this is clear from the definition of boundary morphism and (\ref{eq:e2gamma}).

In the swallowtail with $e^1_\beta = e^1_X$  for $X = S$ or $T$, we have
\[
f(e^1_X \rightarrow e^2_X) = f(\sigma_X) =  f(\sigma_0 * \sigma_1) = f(\sigma_1) \circ f(\sigma_0) = (f(e^1_X \rightarrow e^1_X) \oplus \mathit{id}_{V_{\mathit{cusp}}} )\circ H_X
\]
where we decomposed $\sigma_X = \sigma_0 * \sigma_1$.  Here, $\sigma_0$ is the part of $\sigma_X$ that starts at $R_s$ or $R_t$ and crosses all of the handleslide arcs that end at the $X$ corner of $e^0_{st}$, and $\sigma_1$ is the remaining portion of $\sigma_X$ that runs parallel to $e^1_X$.   The map $H_X$ is as defined in 
 (\ref{eq:HXdef}).  That $f(\sigma_0)$ agrees with $H_X$ is a consequence of the arrangement of handleslide arcs at $e^0_{st}$ specified by Definition \ref{def:nice} (2).
\end{proof}

\subsection{Constructing a CHD from a MC2F} 

\begin{definition}  \label{def:1skeleton}
We say that a CHD $\mathcal{D}= (\{d_\alpha\},\{f_\beta\}, \{K_\gamma\})$ for $\mathcal{E}$ and a nice MC2F $\mathcal{C}=(\{d_\nu\}, H, H_{-1})$ {\bf agree on the $1$-skeleton} 
if for every $0$-cell, $e^0_\alpha$,  and every $1$-cell, $e^1_\beta$,
\begin{equation} \label{eq:dafb}
d_\alpha = d(e^0_\alpha \rightarrow e^0_\alpha)  \quad \mbox{and} \quad f_{\beta} = f(e^1_\beta \rightarrow e^1_\beta),
\end{equation}
where $d(e^0_\alpha \rightarrow e^0_\alpha)$ and $f(e^1_\beta \rightarrow e^1_\beta)$ denote the differentials and continuation maps associated to $0$-cells and $1$-cells by $\mathcal{C}$.  
\end{definition}

\begin{proposition}  \label{prop:CM2FtoCHD}
Let $\mathcal{E}$ be a compatible polygonal decomposition for $L$.  For any nice $\rho$-graded MC2F $\mathcal{C} = (\{d_\nu\},H,H_{-1})$, there exists a $\rho$-graded CHD $\mathcal{D}= (\{d_\alpha\}, \{f_\beta\}, \{K_\gamma\})$ such that $\mathcal{C}$ and $\mathcal{D}$ agree on the $1$-skeleton.
\end{proposition}

\begin{proof}
Use (\ref{eq:dafb}) to define $\{d_\alpha\}$ and $\{f_\beta\}$.  The requirements of Definition \ref{def:CHD} (1) and (2) are easily seen to hold.  In particular, Proposition \ref{prop:dmaps} shows that the $\{f_\beta\}$ have the correct complexes for their domains and codomains, and Proposition \ref{prop:ContinuationMap} (4) shows that the $f_\beta-\mathit{id}$ is strictly upper triangular with degree $0$ mod $\rho$.

It remains to construct the homotopy operators $\{K_{\lambda}\}$. For a given $2$-cell, $e^2_\gamma$, recall the chain maps from  Definition \ref{def:CHD} (3), written there as $\widehat{f_j}^{\eta_j} \circ \cdots \circ \widehat{f_1}^{\eta_1}$ and $\widehat{f_m}^{\eta_m}\circ \cdots \circ \widehat{f_{j+1}}^{\eta_{j+1}}$.  Using Proposition \ref{prop:fmaps} (2), the definition of the $f(e^1_\beta \rightarrow e^2_\gamma)$, and Proposition \ref{prop:ContinuationMap} (5) and (6), we compute 
\begin{align*}
\widehat{f_j}^{\eta_j} \circ \cdots \circ \widehat{f_1}^{\eta_1} & =  f(e^1_j \rightarrow e^2_\gamma)^{\eta_j} \circ \cdots \circ f(e^1_1 \rightarrow e^2_\gamma)^{\eta_1}  \\
 & = f(\sigma_1^{\eta_1} * \cdots * \sigma_j^{\eta_j}) = f(\sigma_a) 
\end{align*} 
where the $\sigma_i$, $1 \leq i \leq j$, are appropriate shifts into $e^2_\gamma$ of the $1$-cells $e^1_{i}$, $1 \leq i \leq j$, that occur around one half of the boundary of $e^2_\gamma$ traversed from $v_0$ to $v_1$.  The concatenation $\sigma_a=\sigma_1^{\eta_1} * \cdots * \sigma_j^{\eta_j}$ is then a shift of this half of the boundary of $e^2_\gamma$ into its interior.  Similarly, $\widehat{f_m}^{\eta_m}\circ \cdots \circ \widehat{f_{j+1}}^{\eta_{j+1}} = f(\sigma_b)$ where $\sigma_b$ is a shift of the other half of the boundary of $e^2_\gamma$.  Since $\sigma_a$ and $\sigma_b$ are path homotopic in the interior of $e^2_\gamma$, Proposition \ref{prop:ContinuationMap} (6) gives the existence of the required (strictly upper triangular) homotopy operator $K_\gamma$.  See Figure \ref{fig:GFtoCHD2}.

\end{proof}

\begin{figure}

\quad

\labellist
\small
\pinlabel $v_0$ [t] at 80 -2
\pinlabel $v_1$ [b] at 120 156
\endlabellist
\centerline{\includegraphics[scale=.6]{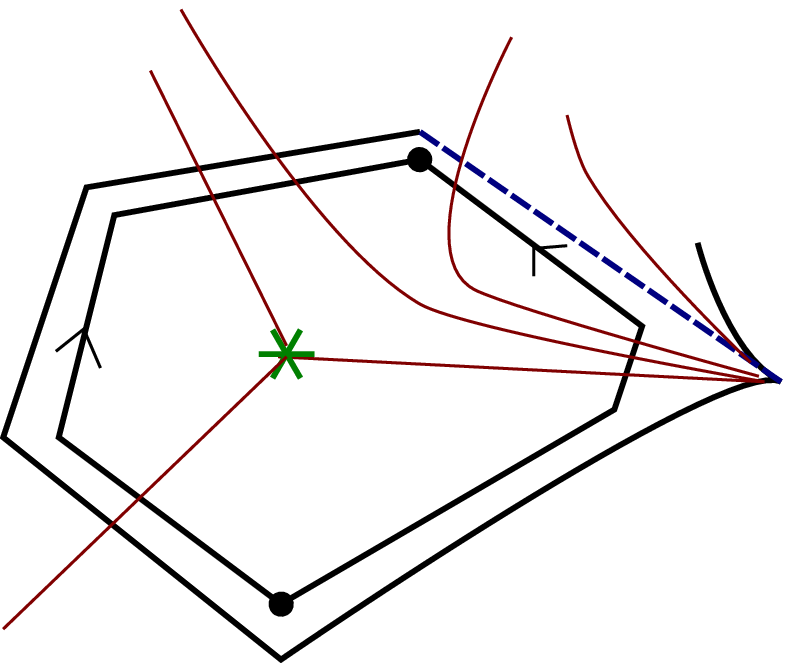}}

\caption{The homotopy operators $K_\gamma$ relate the continuation maps associated to  paths $\sigma_a$ and $\sigma_b$ that trace the boundary of $e^2_\gamma$ from $v_0$ to $v_1$.  The pictured $2$-cell has a swallowtail point at its right-most vertex.}
\label{fig:GFtoCHD2}
\end{figure}

\section{From CHD to MC2F}  \label{sec:CHDMC2F}

We next establish the construction, converse to that of the previous section, of a MC2F from a CHD.  Loosely, this can be viewed as a $2$-dimensional analog of factoring an upper-triangular matrix into a product of elementary matrices.  After observing that this completes the proofs of Theorem \ref{thm:Main2}, we use the connection between CHDs and MC2Fs to associate continuation maps to augmentations.  In Proposition \ref{prop:ObstructA}, we observe that properties of these continuation maps can obstruct the existence of linear at infinity generating families as well as generating families with trivial bundles as their domain.

\subsection{Lemmas for constructing MC2Fs} 

When constructing MC2Fs it is convenient to begin by specifying the handleslide sets $H$ and $H_{-1}$, and then check that the required differentials $d_{\nu}: R_{\nu} \rightarrow R_\nu$ can be constructed, satisfying Axiom \ref{ax:iso}.  We record in Propositions \ref{prop:Extend1}-\ref{prop:Extend3} several cases in which the existence of the differentials is automatic.  See Figure \ref{fig:Extend}.

\begin{proposition}  \label{prop:Extend1} 
Let $L \subset J^1M$ have an MC2F $\mathcal{C}$ defined near the boundary of a disk $D \subset M$ such that $D \cap \Sigma_{\mathit{cusp}} = \emptyset$, where $\Sigma_{\mathit{cusp}}$ is the base projection of cusp edges. Suppose that the handleslide set $H$ of $\mathcal{C}$ is extended over $D$ so that
\begin{itemize}
\item there are no super-handleslide points in $D$, and 
\item Axiom \ref{ax:endpoints} holds.  
\end{itemize}
Then, there is a unique way to assign differentials $d_\nu$ to the regions of $D \setminus \Sigma_\mathcal{C}$, so that $\mathcal{C}= (\{d_\nu\}, H, H_{-1})$ is an MC2F over $D$.  
\end{proposition} 

\begin{proof}
Let $f:(D, \partial D) \rightarrow ([0,1], \{0\})$ be a Morse function with a single critical point that is an absolute maximum at a point $x_0 \in D \setminus \Sigma_{\mathcal{C}}$ with $f(x_0)=1$, and such that the restriction of $f$ to $\Sigma_\mathcal{C}$ is Morse.   It suffices to show how to extend the assignment of differentials $\{d_\nu\}$ from $f^{-1}([0, a-\delta])$ to $f^{-1}([0,a+\delta])$ when $f^{-1}(\{a\})$ contains a single point $p$ that is a codimension $2$ (in $M$) point of $\Sigma_\mathcal{C}$ or a critical point of $f$ restricted to the $1$-dimensional strata of $\Sigma_{\mathcal{C}}$.  Since there are no swallowtails, cusps, or super-handleslides in $D$, we only need to consider:
\begin{itemize}
\item[(a)]  Critical points (max/min) of $f$ restricted to a crossing or handleslide arc.
\item[(b)]  Transverse intersections of two crossing and/or handleslide arcs.    
\item[(c)] Triple points of $\pi_Z(L)$.
\end{itemize} 

Parametrize a neighborhood $N$ of $p$ by $[-\delta,\delta] \times[-\delta, \delta]$ so that $f(x_1,x_2) = a+ x_2$, and  all crossings/handleslides exit $N$ along $x_2 = \pm \delta$.  Let $R_\pm$ denote the regions of $f^{-1}([0,a+\delta]) \setminus \Sigma_\mathcal{C}$ that contain the boundaries $x_1 = \pm \delta$.  Differentials for $R_\pm$ and for all regions in $f^{-1}([0,a]) \setminus \Sigma_\mathcal{C}$ are already specified at the bottom of $N$ where $x_2 = -\delta$.  At  $x_2 = +\delta$, as $x_1$ increases from $-\delta$ to $+\delta$, we pass through a sequence of regions $R_0, R_1, \ldots, R_n$ with $R_0 = R_-$ and $R_n = R_+$.  Since we already have a differential on $R_0$, Axiom \ref{ax:iso} specifies a unique way to assign differentials to $R_1, \ldots, R_{n-1}$.  We just need to verify that the differential on $R_{n-1}$ is related to the one already specified on  $R_n=R_+$ as required in Axiom \ref{ax:iso}.  This amounts to the statement that the continuation map associated to the paths from $R_-$ to $R_+$ at $x_2 = -\delta$ and $x_2 = +\delta$ agree, and this  has already been observed in the proof of Proposition \ref{prop:ContinuationMap}.  



\end{proof}

\begin{proposition}  \label{prop:Extend2}
Suppose that 
 near a swallowtail point $p$ for a Legendrian $L \subset J^1M$, an arbitrary upper triangular differential $d_0$ is assigned to the complement of the swallowtail region, and handleslide arcs, $H$, as required in Axiom \ref{ax:endpoints} (3) are placed within the swallowtail region.
 Then, there exists a unique way to assign differentials $d_\nu$ within the swallowtail region to extend $d_0$ and $H$ to an MC2F $\mathcal{C} = (\{d_\nu\}, H, H_{-1})$ defined near $p$.   
\end{proposition}

\begin{proof}
As usual we consider the case of an upward swallowtail point involving sheets $k$, $k+1$, and $k+2$.  Let $R_0$ be the region with two fewer sheets.  Suppose that as we pass through the swallowtail region from one cusp edge to the other the regions $R_1, \ldots, R_r$ appear in order. 
Passing from $R_0$ into $R_1$, the differential $d_1$ is specified by $d_0$ via Axiom \ref{ax:iso} (3); passing from $R_i$ to $R_{i+1}$ for $1 \leq i \leq r-1$, $d_{i+1}$ is specified by Axiom \ref{ax:iso} (1) and (2).  Finally, when passing from $R_{r}$ back into $R_0$, it is important to have that 
$d_r$ and $d_0$ are related as in Axiom \ref{ax:iso} (3), i.e. we need $d_r = d_0 \oplus d_{k,k+1}$ where $d_{k,k+1}S_{k+1} = S_k$.   
The net effect of passing from $R_1$ to $R_r$ is to conjugate the differential $d_1= d_0 \oplus d_{k,k+1}$ by $H_S \circ Q \circ H_T$ where $Q$ interchanges $S_{k+1}$ and $S_{k+2}$ and the maps $H_S$ and $H_T$ are as in (\ref{eq:HXdef}).  Thus, the required equation is 
\[
(d_0 \oplus d_{k,k+1}) \circ( H_S \circ Q \circ H_T) = ( H_S \circ Q \circ H_T) \circ (d_0 \oplus d_{k,k+1}).
\]
This is straightforward to verify with a direct computation.  Alternatively, observe that if $d_0$ has matrix $A$, then in the notation of Lemma \ref{lem:XST}  the matrix of $d_0 \oplus d_{k,k+1}$ is $\widehat{A}_{k,k+1}$.  The matrices $A_S$ and $A_T$ considered in that lemma have $A_SQ = Q A_T $ (by (\ref{eq:dSdef})), and so using the equation (\ref{eq:Akk1}) we compute
\[
\widehat{A}_{k,k+1} H_S  Q  H_T = H_S A_S Q H_T = H_S Q A_T H_T = H_SQH_T \widehat{A}_{k,k+1}. 
\] 
\end{proof}

\begin{figure}

\quad

\labellist
\small
\pinlabel $d_\nu=$ [r] at 140 138
\pinlabel $d_0$  at 44 318
\pinlabel $d_0$  at 320 318
\pinlabel $d_0$  at 620 324
\pinlabel $d_0$  at 890 324
\endlabellist
\centerline{\includegraphics[scale=.5]{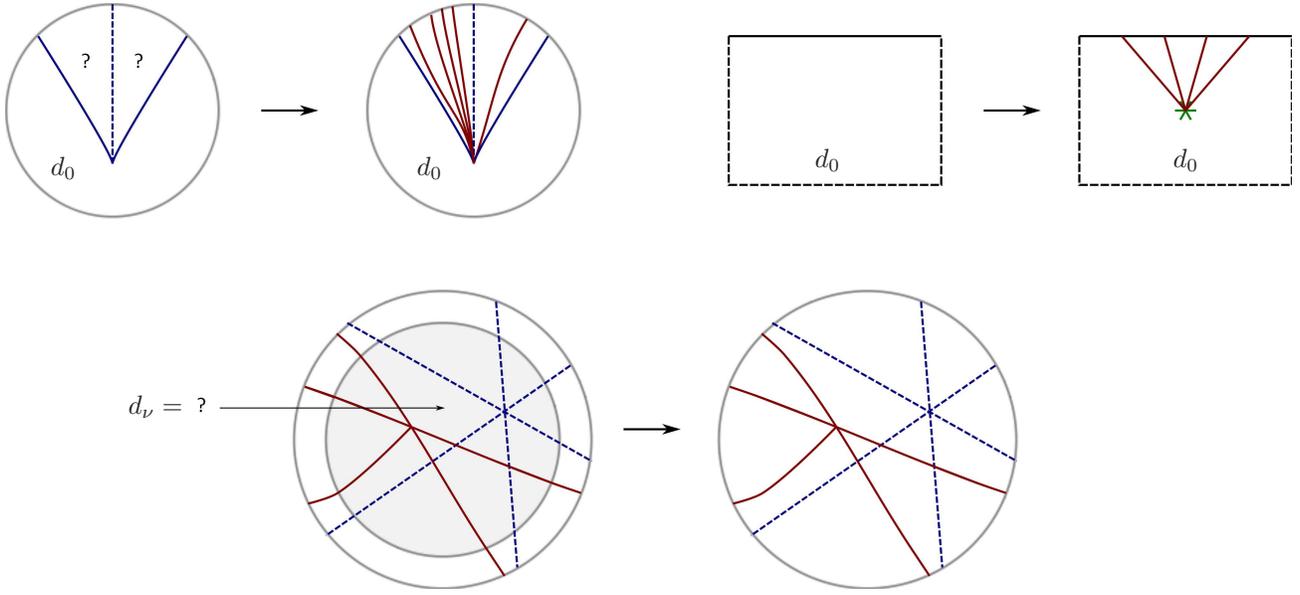}}

\caption{Tools for constructing MC2Fs.  (Clockwise from top left)  Determining differentials near swallowtail points (Proposition \ref{prop:Extend2});   adding superhandleslide points (Proposition \ref{prop:Extend3}); and extending the assignment of differentials $\{d_{\nu}\}$ over the interior of a disk disjoint from $\Sigma_{\mathit{cusp}}$ (Proposition \ref{prop:Extend1}).}
\label{fig:Extend}
\end{figure}

Suppose that an MC2F $\mathcal{C}'$ for $L \subset J^1M$ has been defined on a sub-surface $M' \subset M$ with non-empty boundary.  Let $D \subset (M' \setminus \Sigma_{\mathcal{C}'})$ be a half-open disk 
with $\partial D \subset \partial M'$.  Suppose that $L$ has $n$ sheets above $D$, and let $d_0$ denote the differential assigned to $D$ by $\mathcal{C}'$.

\begin{proposition}  \label{prop:Extend3}
 Suppose that for some $1 \leq i< j \leq n$, we place an $(i,j)$-super handleslide point $p$ in the interior of $D$, and add handleslide arcs in $D$ from $p$ to $\partial D$ as specified by Axiom \ref{ax:endpoints} (2) using the differential $d_0$.   Then, there is a unique way to assign differentials $\{d_\nu\}$ in $D$ to produce an MC2F, $\mathcal{C}$, that agrees with $\mathcal{C}'$ outside of $D$.  
\end{proposition}

\begin{proof}
Again, Axiom \ref{ax:iso} gives a unique way to assign differentials as we pass from $R_0$, the unbounded region of $D$, 
(see Figure \ref{fig:Extend}) through the sequence of new regions $R_1, \ldots, R_r$ created by the handleslides with endpoints at $p$.  We need to verify that Axiom \ref{ax:iso} holds when we pass from $R_r$ back to $R_0$, i.e. that the composition of the handleslide maps associated to the sequence of arcs coming out of $p$  commutes with  $d_0$.  For an $(i,j)$-super handleslide, the matrix for this composition of handleslide maps is 
\[
H= I + D_0 E_{i,j} + E_{i,j}D_0
\]
 where $D_0$ is the matrix of $d_0$, and we compute
 \[
 D_0H= D_0+D_0E_{i,j}D_0 = HD_0.
 \]
\end{proof}

\subsection{Constructing an MC2F from a CHD}

Let $\mathcal{E}$ be a compatible polygonal decomposition for a Legendrian $L \subset J^1M$.
\begin{proposition}
\label{prop:CHDMorse2}
For any $\rho$-graded CHD $\mathcal{D} = (\{d_\alpha\},\{f_\beta\}, \{K_\gamma\})$ for $(L,\mathcal{E})$, there exists a nice $\rho$-graded MC2F $\mathcal{C} = (\{d_\nu\}, H, H_{-1})$  such that $\mathcal{D}$ and $\mathcal{C}$ agree on the $1$-skeleton. 
\end{proposition}

\begin{proof}

\medskip

\noindent{\bf Step 1:}  Defining $\mathcal{C}$ in a neighborhood of the $0$-skeleton.

Let $N_0 \subset M$ consist of a union of small disks, $N_0 = \cup_\alpha N(e^0_\alpha)$, centered at the $0$-cells of $\mathcal{E}$.  Given $e^0_\alpha$, we define $\mathcal{C}$ on $N(e^0_\alpha)$ as follows.

\begin{itemize}
\item When $e^0_\alpha$ is {\it not a swallowtail point}:  We do not introduce any handleslide arcs in $N(e^0_\alpha)$, so we just need to define differentials $d_\nu:V(R_\nu) \rightarrow V(R_\nu)$ for each of the regions $R_\nu \subset N(e^0_\alpha)$ in the complement of the singular set of $L$.  For such a $R_\nu$, we use the usual splitting 
\[
V(R_\nu) = V(e^0_\alpha) \oplus V_{\mathit{cusp}} \quad \mbox{and put} \quad d_\nu = d_\alpha \oplus d_{\mathit{cusp}}.
\]
It is easy to check that Axiom \ref{ax:iso} holds.

\item When $e^0_\alpha$ is a {\it swallowtail point}:  Take the differential $d_0:= d_\alpha$ for the region $R_0$ outside the swallowtail region.  Next, add handleslide arcs as specified by Axiom \ref{ax:endpoints} (3), positioned in the $S$ and $T$ corners as in Definition \ref{def:nice} (2). 
By Proposition \ref{prop:Extend2}, there exists a unique way to define the differentials $d_\nu$ for the components $R_\nu$ of $N(e^0_\alpha) \setminus \Sigma_\mathcal{C}$ within the swallowtail region.

\end{itemize}

\medskip

\noindent{\bf Step 2:}  Extending $\mathcal{C}$ to a neighborhood of the $1$-skeleton. 

Let $N_1$ be the union of $N_0$ with small tubular neighborhoods, $N(e^1_\beta)$, of each $1$-cell.  (In particular, at each swallow tail point $e^0_{st}$,  the $N(e^1_L)$, $N(e^1_R)$, and $N(e^1_{cr})$ should meet the boundary of the disk neighborhood $\partial N(e^0_{st})$ along an arc that is disjoint from the handleslide set of $N(e^0_\alpha)$.)  Given $e^1_\beta$, we now extend $\mathcal{C}$ over $N(e^1_\beta) \setminus N_0$.  Begin by labeling the sheets of $L(e^1_\beta)$ as $S_1, S_2, \ldots, S_n$, and factor $f_\beta$ into a product of handleslide maps
\begin{equation}  \label{eq:hijfactor}
f_\beta = h_{i_r,j_r} \circ \cdots \circ h_{i_1,j_1}.
\end{equation}
(Such a factorization exists by the usual Gauss-Jordan elimination algorithm.)  In $N(e^1_\beta) \setminus N_0$, we then place a sequence of $r$ corresponding handleslide arcs that run across $N(e^1_\beta)$ perpendicularly to $e^1_\beta$; following the orientation of $e^1_\beta$, the 
lower and upper lifts of the $l$-th arc are the sheets above $N(e^1_\beta)$ that continuosly extend $S_{i_l}, S_{j_l}$.  

Starting from the neighborhood of $e^0_-$ where differentials for $\mathcal{C}$ are already defined and following the orientation of $e^1_\beta$ there is a unique way to assign differentials $\{d_\nu\}$ to the regions of $N(e^1_\beta) \setminus \mathcal{C}$ so that Axiom \ref{ax:iso} holds.  Moreover, the factorization (\ref{eq:hijfactor}) shows that when the disk neighborhood of $e^0_+$ is reached the differentials match the previously defined differentials from Step 1.

It is clear at this point that $\mathcal{C}$ agrees with $\mathcal{D}$ on the $1$-skeleton.

\medskip

\noindent{\bf Step 3:}  Extending $\mathcal{C}$ to the interior of $2$-cells.

Given a $2$-cell $e^2_\gamma$, we currently have $\mathcal{C}$ defined in a collar neighborhood, $U \subset \overline{e^2_\gamma}$, of $\partial \overline{e^2_\gamma}$.  Let $C = (\partial U) \cap e^2_\gamma$, i.e. $C$ is a closed curve that is the one boundary component of $U$ belonging to the interior of $e^2_\gamma$.  Let $w_0,w_1 \in C$ denote points on $\partial N(e^0_{v_{i}}) \cap C$ corresponding to the initial and terminal vertices, $v_0$ and $v_1$, of $e^2_\gamma$.  In the case $v_i$ is a swallowtail point where the $S$ or $T$ corner appears in $e^2_\gamma$, place $w_i$ on the $e^1_{cr}$ side of the handleslide arcs that meet $\partial N(e^0_{st})$.
   There are two arcs $\sigma_a$ and $\sigma_b$ oriented from $w_0$ to $w_1$ and such that $C = \sigma_a \cup \sigma_b$.  Along these arcs a sequence of handleslides from $N_1$ meet $C$ transversally, and by construction the continuation maps are
\[
f(\sigma_a) = \widehat{f_j}^{\eta_j} \circ \cdots \circ \widehat{f_1}^{\eta_1} \quad \mbox{and} \quad f(\sigma_b) = \widehat{f_m}^{\eta_m}\circ \cdots \circ \widehat{f_{j+1}}^{\eta_{j+1}}
\]
where we follow the notation of Definition \ref{def:CHD}.    
[This uses that at any swallowtail vertices of $e^2_\gamma$, the handleslide arcs with endpoints on $\partial N(e^0_{st})$ produce the factor of $H_X$ that is required in the definition of boundary map for the edges $e^1_X$ with $X=S$ or $T$.]

The homotopy operator $K_\gamma: (V(e^2_\gamma), \widehat{d}_{v_0})  \rightarrow (V(e^2_\gamma), \widehat{d}_{v_1})$ from $\mathcal{D}$ then satisfies
\[
f(\sigma_a) - f(\sigma_b) = d_{w_1} K_\gamma +K_\gamma d_{w_0}
\]
where the differentials $d_{w_i}$ are from $\mathcal{C}$ at the regions $R_{w_i}$ bordering the $w_i$; they agree with the boundary differentials $\widehat{d}_{v_i}$ written above with the domain and codomain of $K_\gamma$ (by Proposition \ref{prop:dmaps}).  Moreover, 
post-composing both sides with $(f(\sigma_b))^{-1}$ leads to the equation
\begin{equation}  \label{eq:Khomotopy}
f(C) - \mathit{id} = d K + K d
\end{equation}
where we orient $C$ as $\sigma_a * \sigma_b^{-1}$; $K$ is the upper-triangular homotopy operator $K = (f(\sigma_b))^{-1}K_\gamma$; and $d = d_{w_0}$.  

For convenience, in the following we parametrize $\overline{e^2_{\gamma}}$ by $I^2 = [0,1] \times[0,1]$ with coordinates $(x_1,x_2) \in I^2$.  Moreover, we  assume that $N_1$ (the current domain of definition of $\mathcal{C}$) is an $\epsilon$-neighborhood of $\partial(I^2)$, and has its boundary curve $C$ oriented clockwise.  Furthermore, we assume all handleslides arcs in $N_1 \cap \overline{e^2_\gamma}$ appear near the left hand boundary in $[0,\epsilon] \times (\epsilon, 1-\epsilon)$.     
 Note that the differential assigned by $\mathcal{C}$ to the common region bordered by the top, bottom and right side of $\partial(I^2)$ is $d=d_{w_0}$.    

To complete the proof, we extend $\mathcal{C}$ over the remainder of $I^2$.  The approach is pictured schematically in  Figure \ref{fig:CHDtoGF}.  We will use the following terminology:
 We say that the handleslide set $H$ is {\bf lexicographically ordered} along an oriented path $\sigma$ if the indices of upper and lower lifts, $(i,j)$, of handleslide arcs that intersect $\sigma$ are weakly increasing along $\sigma$ with respect to lexicographical order.  We say that two handleslide arcs {\bf commute} if the indices of their lifts, $(i_1,j_1)$ and $(i_2,j_2)$, satisfy $j_1 \neq i_2$ and $i_1 \neq j_2$.

\begin{itemize}

\item  In $[\epsilon, 1/4] \times I$, we extend the handleslide arcs from left to right,  changing their vertical ordering as we go (observing, Axiom \ref{ax:endpoints}), so that $H$ becomes lexicographically ordered along $\{1/4\} \times I$ (as $x_2$ increases). 

[This is possible:  Start by extending the handleslide arcs that begin at $\{\epsilon\} \times I$ to $\{1/4\} \times I$, achieving the required permutation by factoring it into transpositions and interchanging adjacent handleslide arcs in a corresponding manner.  With this initial step carried out, we return to any points where an $(i,l)$-handleslide arc crosses an $(l,j)$-handleslide arc for some $1 \leq i < l < j \leq n$, and for each such point, $x$, create a new $(i,j)$-handleslide arc with one endpoint at $x$ and the other at an appropriate point on $\{1/4\} \times I$.  Repeat this procedure inductively.  Note that any $(i,j)$-handleslide arc created at the $m$-th step will have $i-j \geq m$, so that after finitely many steps the process is complete.]

For any $i<j$, let $\alpha_{i,j}$ 
 be the number of $(i,j)$-handleslide arcs at $\{1/4\} \times I$.  We can arrange that each $\alpha_{i,j}$ is either $0$ or $1$ 
since an adjacent pair of $(i,j)$-handleslide arcs with endpoints at $\{1/4\} \times I$ can be joined together into a single arc with a local maximum for the $x_1$-coordinate just before $x_1 = 1/2$.     The continuation map for $\{1/4\} \times I$ agrees with $f(C)$ (by Proposition \ref{prop:ContinuationMap} (6)), and by definition is 
\[
f(C) =  h_{n-1,n}^{\alpha_{n-1,n}}(h_{n-2,n}^{\alpha_{n-2,n}} h_{n-2,n-1}^{\alpha_{n-2,n-1}}) \cdots (h_{2,n}^{\alpha_{2,n}}\cdots h_{2,3}^{\alpha_{2,3}})(h_{1,n}^{\alpha_{1,n}} \cdots h_{1,2}^{\alpha_{1,2}}).
\]
Observe that (due to the lexicographic ordering of subscripts) the matrix of this product is precisely
\[
I+ \sum_{i<j} \alpha_{i,j} E_{i,j},
\]
so  
\[
\alpha_{i,j} = \langle (f - \mathit{id}) S_j, S_i \rangle.
\]


\item  In $[3/4, 1] \times I$, we start by placing in lexicographic order at $x_2 = 7/8$ an $(i,j)$-super handleslide point, for each $i<j$ with  $\langle K S_{j}, S_{i}\rangle=1$.
In addition, we add handleslide arcs as specified by Axiom \ref{ax:endpoints} (2) running approximately horizontally from $\{7/8\} \times I$ to $\{3/4\} \times I$.   
As in Observation \ref{ob:CM2F} (1), we can always use the differential $d=d_{w_0}$ in determining what (if any) handleslide arcs need to appear with endpoint at a super-handleslide.  It follows, at least mod $2$, that the total number of $(i,j)$-handleslide arcs along $\{3/4\} \times I$ is 
\[
\langle (d K + K d) S_j , S_i \rangle = \langle (f - \mathit{id})S_j, S_i \rangle = \alpha_{i,j}.
\]

By Proposition \ref{prop:Extend3}, there is a unique way to assign differentials in $[3/4, 1] \times I$ to any new regions that are created by the handleslides ending at the new super-handleslide points.  

\item In $[1/2,3/4] \times I$, we extend the handleslide arcs from $x_1=3/4$ to $x_1=1/2$, arranging that the handleslides are 
lexicographically ordered at $\{1/2\} \times I$.  Moreover, this can be done {\it without creating additional handleslide endpoints.}  

[Assume inductively that the subset $X_{<m}$ of handleslide arcs that have their right endpoint at an $(i,j)$-superhandleslide points with $i< m$ have been extended to $\{1/2\}\times I$ where they   
appear in lexicographic order.   
To inductively complete the extension process, we need to extend the subset $X_m$ of those handleslide arcs with right endpoint at an $(m,j)$-super-handleslide.  
Any such arc in $X_m$ will be an $(i',j')$-handleslide for with $i'\leq m < j \leq j'$.  Consequently, arcs in $X_m$ commute with one another.     At $x_1=3/4$, all handleslide arcs from $X_{<m}$ appear below the arcs from $X_m$.  Consequently, to extend a given $(i',j')$-handleslide arc from $X_m$ appropriately, it will only need to cross $(i'',j'')$-handleslides from $X_{<m}$ having $ i'\leq i''$.  In these cases the $(i',j')$ and $(i'',j'')$ are such that the arcs commute since $i''\leq m < j'$ (because $(i'',j'')$ has an endpoint at an $(i,j)$-superhandleslide with $i < m$) and $i' \leq i'' <j''$.]     

Since no new handleslide arcs were created, the number of $(i,j)$-handleslide arcs at $\{1/2\} \times I$ is still $a_{i,j}$ mod $2$, and joining $(i,j)$-handleslide arcs together in pairs, we can assume the number of arcs is exactly $\alpha_{i,j}$.  


\item In $[1/4,1/2]\times I$, since handleslide arcs are lexicographically ordered along $x_1=1/4$ and $x_1=1/2$ and are in bijection (preserving $(i,j)$), we simply join the end points.  

\end{itemize}

With the handleslide set complete, Proposition \ref{prop:Extend1} shows that the differentials $\{d_\nu\}$ can be defined over $[\epsilon, 3/4] \times I$.  This completes the construction of $\mathcal{C}$.

\end{proof}

\begin{figure}
\labellist
\small
\pinlabel $(1,2)$  at 196 48
\pinlabel $(1,n)$  at 196 96
\pinlabel $\vdots$  at 196 80
\pinlabel $\vdots$  at 196 156
\pinlabel $\vdots$  at 196 228
\pinlabel $(2,3)$  at 196 128
\pinlabel $(2,n)$  at 196 176
\pinlabel $(n-1,n)$  at 196 272
\pinlabel $(1,2)$ [l] at 400 48
\pinlabel $(1,n)$ [l] at 400 96
\pinlabel $(2,3)$ [l] at 400 128
\pinlabel $\vdots$ [l]  at 408 156
\pinlabel $\vdots$ [l] at 408 80
\pinlabel $\vdots$  [l] at 408 228
\pinlabel $(n-1,n)$ [l] at 396 272
\endlabellist
\centerline{\includegraphics[scale=.6]{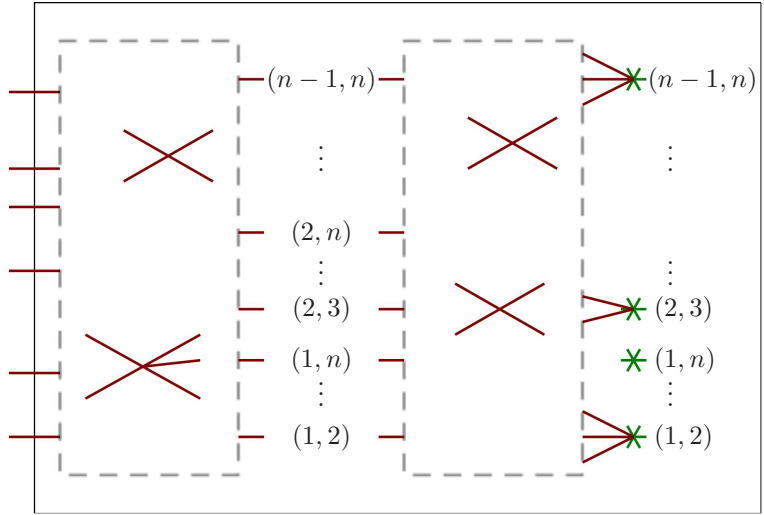}}

\caption{Extending $\mathcal{C}$ over the interior of $e^2_\gamma$.}
\label{fig:CHDtoGF}
\end{figure}

Theorem \ref{thm:Main2} that was stated in the introduction now follow easily.

\begin{proof}[Proof of Theorem \ref{thm:Main2}]
Proposition \ref{prop:AugmentationCHD} shows the existence of a $\Z/2$-augmentation is equivalent to the existence of a CHD.  
Since  a small perturbation can make any MC2F nice with respect to a given $\mathcal{E}$, Proposition \ref{prop:CM2FtoCHD} and  Proposition \ref{prop:CHDMorse2} show that $L$ has a CHD if and only if $L$ has a MC2F.  
 The statement about generating families then follows from Proposition \ref{prop:GFMorse2}.
\end{proof}

\subsection{Monodromy representations for augmentations}  \label{sec:MonoAug}

Using Proposition \ref{prop:CHDMorse2}, we can now associate a fiber homology space with monodromy representation to an augmentation.

Let $\mathcal{E}$ be a compatible polygonal decomposition for $L$, and let $\epsilon: (\mathcal{A}, \partial) \rightarrow (\Z/2, 0)$ be an augmentation of the corresponding Cellular DGA.  Let $e^0_\alpha \in \mathcal{E}$ be a $0$-cell.  Consider a small neighborhood $N(e^0_\alpha)$, and let $x_0 \in N(e^0_\alpha)$ be disjoint from the cusp/crossing locus; if $e^0_\alpha$ is a swallowtail point, we assume $x_0$ is outside the swallowtail region. 
Via Proposition \ref{prop:AugmentationCHD}, there is a unique CHD, $\mathcal{D}$, for $\mathcal{E}$ associated to $\epsilon$.  Then, using Proposition \ref{prop:CHDMorse2} there exists an MC2F $\mathcal{C}$ that agrees with $\mathcal{D}$ on the $1$-skeleton.  We can assume the handleslide set of $\mathcal{C}$ is disjoint from $N(e^0_\alpha)$, or the part of $N(e^0_\alpha)$ outside the swallowtail region in the case $e^0_\alpha$ is a swallowtail.    

We define the {\bf fiber homology} and {\bf monodromy representation} of $\epsilon$ at $x_0$, by
\[
H(\epsilon_{x_0}):= H(\mathcal{C}_{x_0}) \quad \mbox{and} \quad \Phi_{\epsilon, x_0}:= \Phi_{\mathcal{C},x_0}:\pi_1(M,x_0) \rightarrow \mathit{GL}(H(\epsilon_{x_0})).
\]
(Recall $H(\mathcal{C}_{x_0})$ and $\Phi_{\mathcal{C},x_0}$ are defined in Corollary \ref{cor:Cx0}.)
\begin{proposition}
For $x_0$ as above, $H(\epsilon_{x_0})$ and $\Phi_{\epsilon, x_0}$ are well-defined.
\end{proposition}
\begin{proof}
Since $\mathcal{C}$ and $\mathcal{D}$ agree on the $1$-skeleton, the differential on $V(R_0)$ (where $R_0 \subset M \setminus \Sigma_{\mathcal{C}}$ and $x_0 \in R_0$) is determined by the differential $d_\alpha$ on $V(e^0_\alpha)$ from $\mathcal{D}$ via the boundary differential construction.

In addition, the continuation maps $f(\sigma)$ for paths $\sigma$ that are shifts of a $1$-cell $e^1_\beta$ into bordering $2$-cells are determined by the map $f_\beta$ from $\mathcal{D}$ via the boundary map construction.  Any $[\sigma] \in \pi_1(M_, x_0)$ can be represented by a concatenation of such paths with some paths, $\tau_i$, contained in the $N(e^0_\alpha)$.  In the swallowtail case, the handleslide set, $H$, of $\mathcal{C}$ has a standard form in the $S$ and $T$ sides of the part of $N(e^2_\alpha)$ in the swallowtail region, while in other cases $H$ is disjoint from $N(e^2_\alpha)$.  Thus, we can take the $\tau_i$ to be independent of $\mathcal{C}$, so that $\Phi_{\mathcal{C},x_0}([\sigma])$ is determined by $\mathcal{D}$.

\end{proof}

\begin{remark}  
As in Remark \ref{rem:localsystem}, although we have only defined $(H(\epsilon_{x_0}), \Phi_{\epsilon, x_0})$ near $0$-cells, up to isomorphism there is a unique local system on all of $M$ extending $(H(\epsilon_{x_0}), \Phi_{\epsilon, x_0})$.  
\end{remark}

\begin{observation}
\begin{enumerate}
\item From Corollary \ref{cor:Cx0}, it follows that the isomorphism type of $H(\epsilon_{x_0})$ is independent of $x_0$.
\item Explicitly, the group $H(\epsilon_{x_0})$ is computed from $\epsilon$ as the homology of $(V(e^0_\alpha), d_\alpha)$ where 
\[
d_\alpha S_{j} = \sum_{i} \epsilon(a^{\alpha}_{i,j}) S_i.
\]
The monodromy map $\Phi_{\epsilon, x_0}([\sigma])$ is computed by homotoping $\sigma$ into a concatenation of $1$-cells, $e^1_{\beta_1}* \cdots * e^1_{\beta_m}$; shifting each such $1$-cell into the interior of a neighboring $2$-cell (as in Section \ref{sec:Shifts}); and then connecting the endpoints with paths $\tau_i$ in the $N(e^0_\alpha)$.  The resulting map has the form
\[
\Phi_{\epsilon, x_0}([\sigma]) = f(\tau_{m}) \circ \widehat{f}_{\beta_m}^{\pm 1} \circ f(\tau_{m-1}) \circ \widehat{f}_{\beta_{m-1}}^{\pm 1} \circ \cdots \circ f(\tau_{1}) \circ \widehat{f}_{\beta_1}^{\pm1} \circ f(\tau_{0})
\]
where each $\widehat{f}_{\beta_i}$ is obtained from the map $f_{\beta_i}$ from $\mathcal{D}$ as in the boundary map construction.  Except in the case of a swallowtail point, the $f(\tau_{i})$ are simply compositions of the projection/inclusion maps, $p$ and $j$, from cusp edges, and the permutation maps from crossings.  At swallowtails, when $\tau_{i}$ connects an endpoint outside of the swallowtail region to one within the $S$ (resp. $T$) region, the map $f(\tau_i)$ is 
\[
\begin{array}{lccc}   & H_S \circ j  & \mbox{or} &  p \circ H_S \\ \mbox{(resp.} &  H_T \circ j  & \mbox{or} &  p \circ H_T) \end{array}
\]
depending on the orientation of $\tau_i$.

\end{enumerate}
\end{observation}

We arrive at the following obstructions to particular types of generating families.

\begin{proposition}  \label{prop:ObstructA}
\begin{enumerate} 
\item If $H(\epsilon_{x_0}) \neq \{0\}$ for all augmentations $\epsilon$, then $L$ does not have a linear at infinity generating family.
\item If $\Phi_\epsilon$ is non-trivial for all augmentations $\epsilon$, then $L$ does not have a generating family whose domain is a trivial bundle over $M$.
\end{enumerate}
\end{proposition}

\begin{proof}
Follows directly from the itemized statements in Proposition \ref{prop:GFMorse2} and the definition of $(H(\epsilon_{x_0}), \Phi_{\epsilon, x_0})$.  
\end{proof}

\section{Examples}
\label{sec:Examples}


An easy corollary of Theorem \ref{thm:Main2} is that loose Legendrian surfaces \cite{Murphy} do not have generating families, since they do not have augmentations.  In this section, we consider further examples, including Legendrians to which the more refined obstructions of Proposition \ref{prop:ObstructA} can be applied.

\subsection{Treumann-Zaslow Legendrians}  In \cite{TZ}, Treumann and Zaslow introduce an elegant class of Legendrian surfaces associated to trivalent graphs.
For these surfaces, they study associated moduli spaces of constructible sheaves, and construct examples of non-exact Lagrangian fillings.  In this section, we apply our approach to provide necessary and sufficient conditions for the existence of $\Z/2$-augmentations for this class of Legendrian surfaces.  

\medskip

Let $\Gamma \subset M$, be a  tri-valent graph.  In Section 2.1 of \cite{TZ}, a front projection called the {\it hyperelliptic wavefront} modeled on $\Gamma$ is constructed\footnote{Strictly speaking, \cite{TZ} considers the case of $M=S^2$, but the construction works equally well to produce a front projection in $J^1M$ for any surface $M$.} producing a Legendrian that we denote $L_\Gamma \subset J^1M$.  The base projection $\pi_B:L_\Gamma \rightarrow M$, is a $2$-fold branched covering of $M$, with branch points at the vertices of $\Gamma$.  Crossing arcs of the front projection sit above the edges of $\Gamma$; above each vertex of $\Gamma$,  $\pi_F(L_\Gamma)$ matches a standard coordinate model in which $3$ crossing arcs share a common endpoint.  
The front singularities that appear above vertices in $\pi_F(L_\Gamma)$ are non-generic, but appear with codimension $1$ in the space of front projections as the $D^-_4$ bifurcation of fronts;  see \cite{ArnoldGuseinZadeVarchenko}.  A generic front for a surface Legendrian isotopic to $L_\Gamma$ is obtained by replacing the singularities above vertices with the configuration of $3$-swallowtail points pictured in Figure \ref{fig:D4}.

\begin{figure}
\labellist
\small
\pinlabel $B_1$ [r]  at 490 90
\pinlabel $B_2$ [l] at 586 90
\pinlabel $B_3$ [l] at 540 90
\pinlabel $B_4$ [bl] at 580 188
\pinlabel $A_1$ [b] at 538 214
\pinlabel $A_0$ [t] at 538 12
\pinlabel $C_1$  at 506 146
\pinlabel $C_2$  at 570 146
\endlabellist
\centerline{\includegraphics[scale=.5]{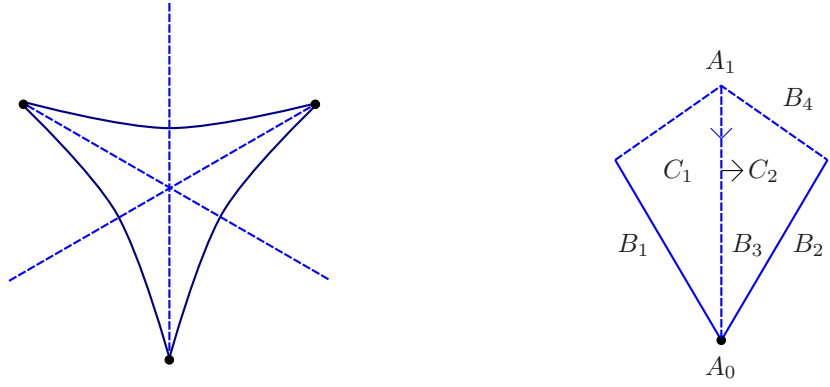}} 

\caption{(left) A generic base projection for $L_\Gamma$ near vertices of $\Gamma$.  There are three upward swallowtail points all placed on the upper sheet of $L_\Gamma$.  (right) Labeling for $0$- and $1$-cells used in proof of Proposition \ref{prop:TZsurface}. }
\label{fig:D4}
\end{figure}








We refer to the components of $M \setminus \Gamma$ as {\bf faces} of $\Gamma$, but note that they do not need to be disks.  
\begin{proposition} \label{prop:TZsurface} The Legendrian surface $L_\Gamma$ has a $\Z/2$-augmentation if and only if every face of $\Gamma$ has an even number of vertices. 
\end{proposition}

\begin{proof}
In a neighborhood $N(v) \subset M$ of any vertex of $v \in \Gamma$, an MC2F $\mathcal{C}$ can be constructed with handleslide set as pictured in Figure \ref{fig:Slice}.   The differentials $\{d_v\}$ are $0$ in regions where $L_\Gamma$ is $2$-sheeted;  Proposition \ref{prop:Extend2} then defines differentials in neighborhoods of swallowtail points, and this assignment of differentials can be extended to a neighborhood of the cusp locus so that for regions bordering the cusp locus the only non-zero $d_\nu S_i$ is $d_\nu S_b = S_a$ with $S_b$ and $S_a$ the lower and upper sheets at the cusp edge.  Finally, Proposition \ref{prop:Extend1} extends the differentials $\{d_v\}$ over the remainder of $N(v)$.  
 
Note that one handleslide arc enters each of the three faces adjacent to $v$.  For any face $F$ with an even number of vertices, it is then easy to extend $\mathcal{C}$ over $F$ by connecting the handleslide arcs that exist near the vertices in pairs via paths in the interior of $F$.   

\begin{figure}
\labellist
\small
\endlabellist
\centerline{\raisebox{1cm}{\includegraphics[scale=.4]{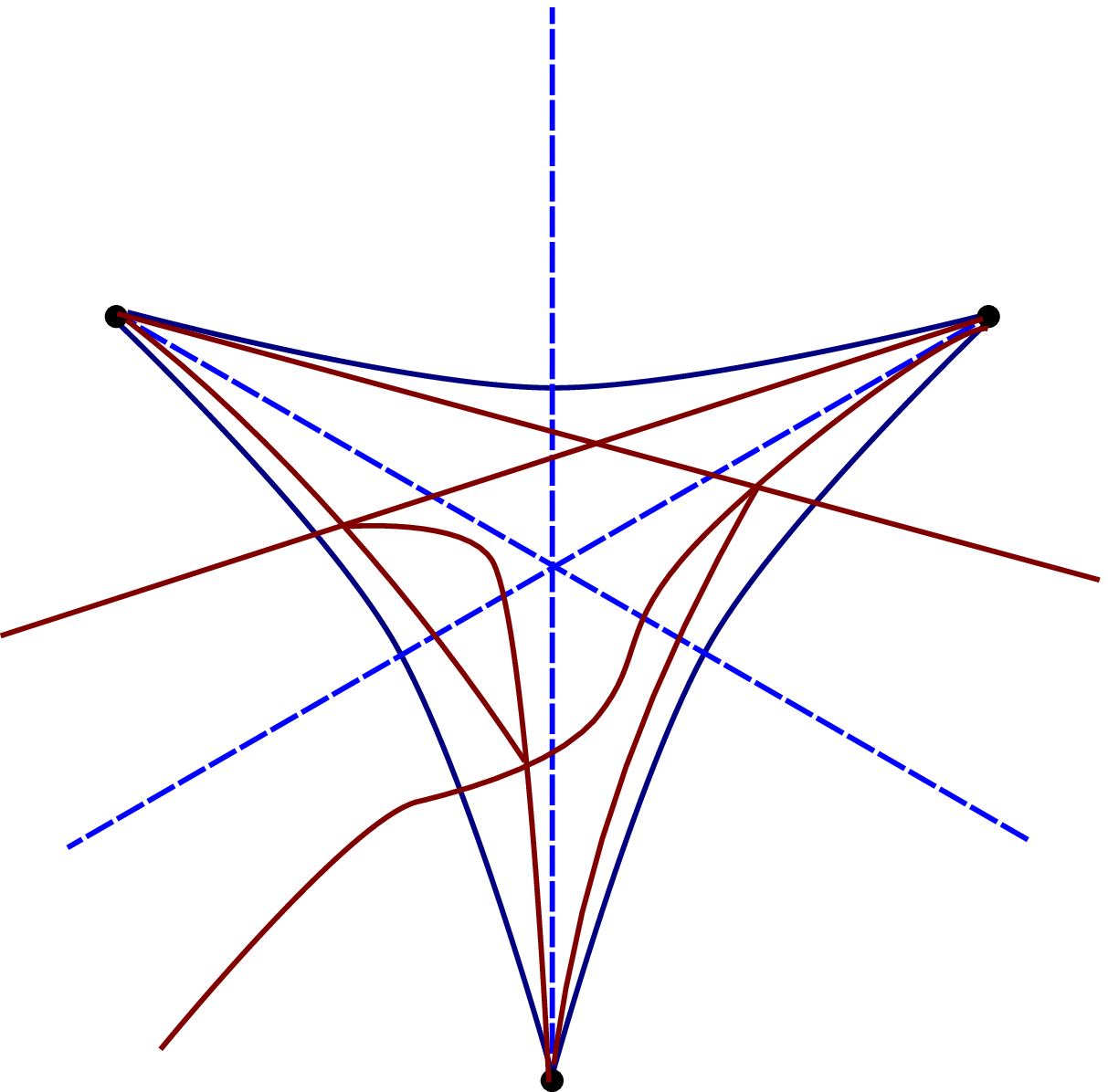}}\quad \quad \quad \quad \includegraphics[scale=.5]{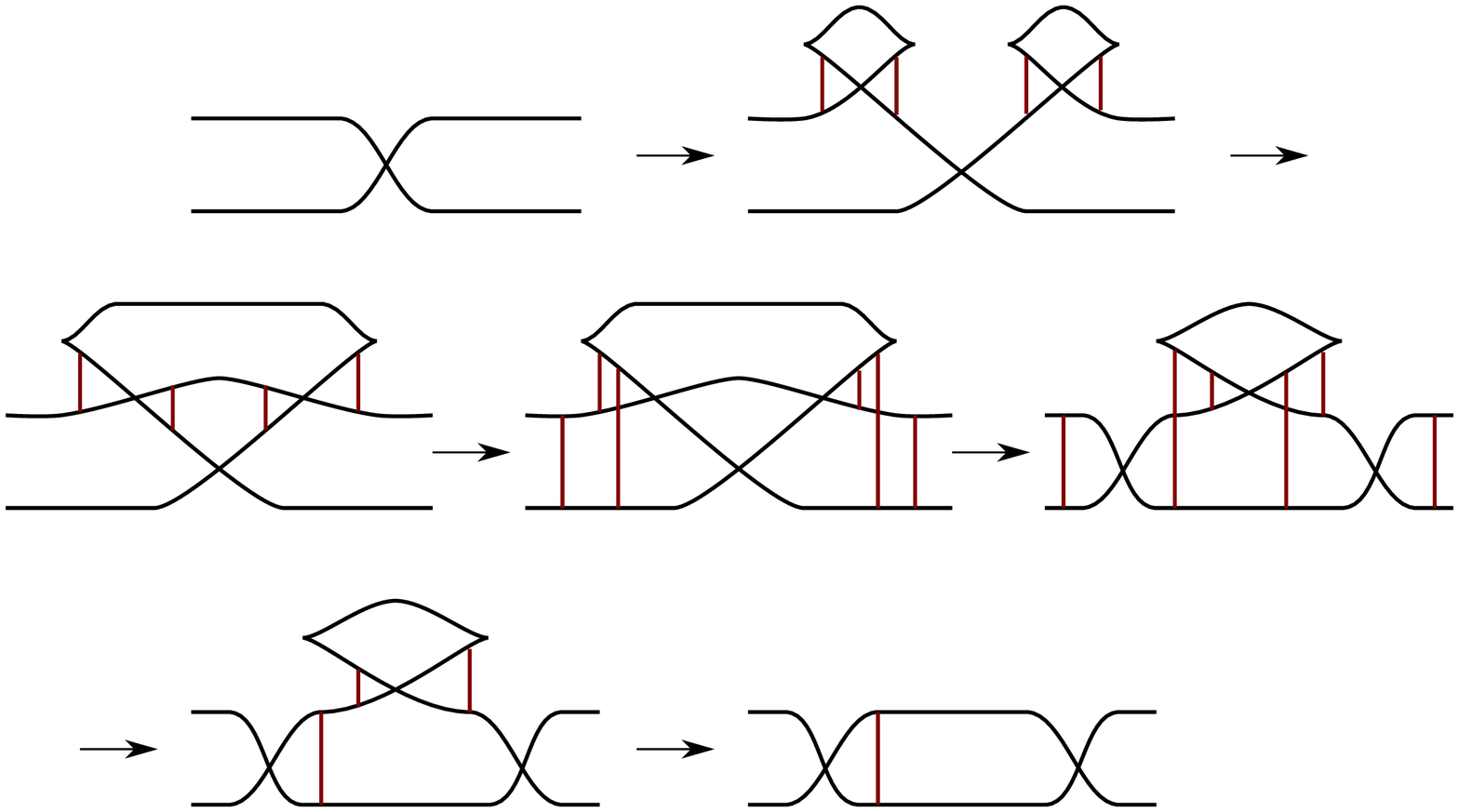}} 

\caption{(left) An MC2F for $L_\Gamma$ near vertices of $\Gamma$.  (right) Slices of the front projection of $L_\Gamma$ as $x_2$ decreases. }
\label{fig:Slice}
\end{figure}

It remains to show it is impossible to construct an MC2F if there is at least one face, $F$, with an odd number of vertex.  For a vertex $v$ of $F$, the neighborhood $N$ of $v$ above which $L_\Gamma$ is $4$-sheeted  has a natural polygonal decomposition with $6$ triangular $2$-cells.  Consider the $1/3$ of $N$ consisting of the two triangles with vertices at the swallowtail point that points into the face $F$, and label the cells of this region as in Figure \ref{fig:D4} (right); number sheets above $A_1$ and $B_3$ as they are ordered above $C_2$.  The choice of $T$ and $S$ corner at $A_1$ is unimportant since $T=S= I+E_{2,3}$.

\medskip

\noindent {\bf Claim:}  Any augmentation $\epsilon : \mathcal{A} \rightarrow \Z/2$ satisfies  $\epsilon(b^1_{1,2}) + \epsilon(b^2_{1,2}) = 1$.  

\begin{proof}[Proof of Claim]
   The differential $\partial B_3 = (A_0)_T (I+B_3) +(I+B_3) A_1$ is
\[
\partial B_3 = \left[\begin{array}{cccc} 0 & 1 & 1 & 0 \\  & 0 & 0 & a^0_{1,2} \\  & & 0 &  a^0_{1,2} \\ & & & 0 \end{array} \right]( I +B_3) + (I+B_3) 
\left[\begin{array}{cccc} 0 & a^1_{1,2} & a^1_{1,3} & a^1_{1,4} \\  & 0 & 0 & 0 \\  & & 0 &  0 \\ & & & 0 \end{array} \right] 
\] 
The $(2,3)$-entry of $B_3$ is $0$ because of the crossing locus, so the top row of $\partial B_3$ is 
\[
[0 \,\, 1+a^1_{1,2} \,\, 1+a^1_{1,3} \,\, a^1_{1,4}+ b^3_{2,4}+b^3_{3,4}].
\]
The equation $\epsilon(\partial B_3)=0$ then implies 
\begin{equation} \label{eq:B3}
\epsilon(a^1_{1,2}) =1;  \quad \epsilon(a^1_{1,3}) =1; \quad \epsilon(a^1_{1,4}) =\epsilon( b^3_{2,4})+\epsilon(b^3_{3,4}).
\end{equation}
If we consider the corresponding equation $\partial B_3'$ where $B_3'$ belongs to a different $1/3$ of $N$, the location of the $a^1_{1,j}$ are permuted within the matrix $A_1$. For instance,  in $\partial B'_3$, the top row of the $A_1$ matrix would become $[ 0\, a^1_{1,2}\, a^1_{1,4} \, a^1_{1,3} ]$ or  $[ 0\, a^1_{1,4}\, a^1_{1,2} \, a^1_{1,3} ]$ depending on the choice of total ordering of sheets above $B'_3$.  Thus, (\ref{eq:B3}) for $B_3'$ gives $\epsilon(a^1_{1,4}) = 1$ as well, so that the last equality of (\ref{eq:B3}) is
\[
1 = \epsilon( b^3_{2,4})+\epsilon(b^3_{3,4}).
\]  
Now, considering the $2\times 2$ block consisting of the $3$-rd and $4$-th rows and columns of $\partial C_1$ (resp. $\partial C_2$) gives the equation
\[
\epsilon(b^3_{2,4}) + \epsilon(b^1_{1,2}) = 0  \quad \quad \mbox{(resp.  } \epsilon(b^3_{3,4})+\epsilon(b^2_{1,2}) = 0),
\]
so $\epsilon(b^1_{1,2}) + \epsilon(b^2_{1,2}) =1$ as claimed.  [For instance, in the equation  
\[
\partial C_2 = A_{v_1}C_2+C_2A_{v_0} + T(I+B_2)(I+B_4)+ (I+B_3),
\]
 note that the $(3,4)$-entry of $T=I+E_{2,3}$, $B_4$,  $A_{v_1}C$, and $CA_{v_0}$ are all zero, since sheets $S_3$ and $S_4$ cross above $B_4$ and the $A$ and $C$ matrices are both strictly upper-triangular.]
\end{proof}

With the claim in hand, we note that if $\mathcal{A}$ had an augmentation then from Proposition \ref{prop:CHDMorse2}, there would exist a MC2F agreeing with $\epsilon$ on the $1$-skeleton and hence having, at each vertex of $F$, an odd number of handleslide arcs crossing into $F$ through the $1$-cells $B_1$ and $B_2$. 
Since no handleslides can enter $F$ along the crossing arcs that run along the edges of $\Gamma$, and $F$ has an odd number of vertices, this means that in total there are an odd number of handleslide arcs entering the $2$-sheeted region above $F$.  But, this is impossible since these arcs would have to meet in pairs in the interior of the $2$-sheeted region of $F$.


\end{proof}

\begin{remark}
For 1-dimensional Legendrian knots, it is shown in \cite{NRSSZ} that the category of constructible sheaves from \cite{STZ} is equivalent to a category whose moduli space of objects consists of augmentations up to DGA homotopy.  A close connection between constructible sheaves and augmentations is expected in general.  

Proposition 1.2  of \cite{TZ}, shows that over $\Z/2$, $L_\Gamma \subset J^1S^2$ has a constructible sheaf defined over $\Z/2$ if and only if the dual graph to $\Gamma$ is $3$-colorable.  When $\Gamma$ is $3$-valent this condition is equivalent to every face of $\Gamma$ having an even number of vertices, so our Proposition  \ref{prop:TZsurface} is consistent with the expected connection between constructible sheaves and augmentations.  A more extensive study of the DGAs for the Treumann-Zaslow fronts, including results about augmentations implying Proposition \ref{prop:ObstructA}, is made in the recent work of Casals and Murphy, \cite{CasalsMurphy}.
\end{remark}

\subsection{The conormal of the unknot}

The unit conormal bundle of the unknot is a Legendrian torus in the unit cotangent bundle $ST^*\R^3$ that, using a canonical contactomorphism $ST^*\R^3 \cong J^1S^2$, becomes a Legendrian $\Lambda_U \subset J^1S^2$.  The front projection of $\Lambda_U$ can be taken to be two sheeted with cone points at $(0,0,1)$ and $(0,0,-1)$ and no other singularities.  A generic front diagram for $\Lambda_U$ is obtained by perturbing the cone points to produce the configuration of $4$ swallowtail points connected with cusps and crossings as pictured in Figure \ref{fig:Cone}.  The four cusp arcs connect the middle two sheets labeled $S_2$ and $S_3$ above the cells inside the swallowtail region.  The vertical (resp. horizontal) crossing arc is between sheets $S_3$ and $S_4$ (resp. sheets $S_1$ and $S_2$) and has its endpoints at two upward (resp. downward) swallowtail points.  See \cite{EENS} for more details.

\begin{figure}
\labellist
\small
\pinlabel $S_1=S_2$ [l]  at 230 170
\pinlabel $S_3=S_4$ [r] at 18 210
\pinlabel $S_0$ [b] at 134 34
\pinlabel $T_0$ [b] at 112 34
\pinlabel $T_1$ [r] at 216 112
\pinlabel $T_2$ [l] at 34 112
\pinlabel $B_0$ [l] at 126 94
\pinlabel $C_1$ at 82 100
\pinlabel $C_2$ at 174 100
\pinlabel $B_4$ [b] at 72 126
\pinlabel $B_3$ [b] at 150 126
\pinlabel $B_2$ [tr] at 70 74
\pinlabel $B_1$ [tl] at 176 74
\pinlabel $A_0$  [t] at 124 0
\endlabellist
\centerline{\raisebox{1cm}{\includegraphics[scale=.6]{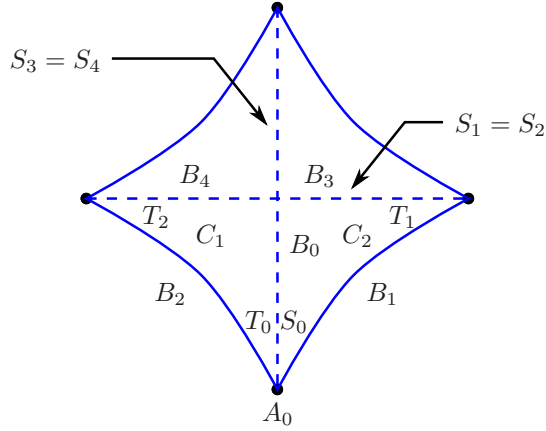}}} 

\caption{The resolution of a cone point, with labeling of cells and choice of $S$ and $T$ corners at swallowtail points as used in the proof of Proposition \ref{prop:conormal}.}
\label{fig:Cone}
\end{figure}

\begin{proposition}  \label{prop:conormal}
The conormal of the unknot $\Lambda_U \subset J^1S^2$ does not have any linear at infinity generating family.
\end{proposition}

\begin{proof}
There is an obvious polygonal decomposition near the resolved cone point, and we label cells as in Figure \ref{fig:Cone}.   For any augmentation, the fiber homology $H(\epsilon_{x_0})$ can be computed from the complex associated to $A_0$ 
 which is
\[
V = \mbox{Span}(S_1,S_2) \quad \mbox{with}  \quad dS_1= 0, \,\, dS_{2} = \epsilon(a^0_{1,2}) S_1.
\]
Thus, the result follows from  Proposition \ref{prop:ObstructA} (1) once we show that $\epsilon(a^0_{1,2})=0$ holds for any $\epsilon$.

To this end, consider the differential of $C_2$, which (using initial and terminal vertices $v_0=v_1=A_0$) is
\[
\partial C_2 = A_{v_0}C + C A_{v_1} + (I+B_0)(I+B_3)T_1(I+B_1)S_0 +I.
\]
Since the matrices are upper-triangular, the same equation holds when considering the upper-left $2\times 2$-blocks; 
 this $2\times 2$-block is 
\[
\left[\begin{array}{cc} 0 & \partial c^2_{1,2}  \\ 0 & 0 \end{array} \right] = 0 + 0 + (I+b^0_{1,2}E_{1,2})(I)(I+E_{1,2})(I)(I+ a^0_{1,2}E_{1,2}) +I = \left[\begin{array}{cc} 0 & b^0_{1,2}+1+a^0_{1,2}  \\ 0 & 0 \end{array} \right].
\]
Thus, $\epsilon \circ \partial =0$ implies
\begin{equation} \label{eq:b12}
\epsilon(b^0_{1,2}) = 1 + \epsilon(a^0_{1,2}).
\end{equation}
[The $(1,2)$-entry of the matrices $B_3$ and $B_1$ are zero because of a crossing and cusp arc respectively.]

Similarly, considering the upper left $2\times 2$-block of 
\[
\partial C_1 = A_{v_0}C + C A_{v_1} + (I+QB_0Q)(I+B_4)T_2(I+B_2)T_0 +I
\]
(the matrix $Q = Q_{3,4}$ is the permutation matrix for $(3 \, 4)$), gives
\begin{equation} \label{eq:b12two}
\epsilon(b^0_{1,2}) = 1.
\end{equation}
Thus, the required equality $\epsilon(a^0_{1,2})=0$ follows from comparing (\ref{eq:b12}) and (\ref{eq:b12two}).

\end{proof}

\begin{remark}
It is interesting to note that any generic $1$-dimensional slice of $\Lambda_U$ {\it does} admit a linear at infinity generating family.  Indeed, pulling the front projection of $\Lambda_U$ back along an immersion $f:S^1 \rightarrow S^2$ that is transverse to the base projection of the singular set produces a Legendrian $\Lambda_f \subset J^1S^1$.  The front projection of any $\Lambda_f$ has a graded normal ruling obtained from taking all crossings to be switches, i.e. above the $4$ sheeted region the middle two sheets are paired as are the outer two sheets.   See \cite{ChP} or \cite{Ru} for a discussion of normal rulings in $J^1S^1$; the proof of equivalence of the existence of  graded normal rulings and linear at infinity generating families from \cite{FuchsRutherford} continues to hold in the $J^1S^1$ setting since away from crossings and cusps the generating families constructed in Section 3 of \cite{FuchsRutherford} have a standard form depending only on the pairing of sheets.  
\end{remark}

\begin{remark}
As an alternate approach, the definition of MC2F and main results of this paper can all be extended to allow fronts with cone point singularities using the extension of the cellular DGA to such fronts given in Section 5.3 of \cite{RuSu1}.  The definition of MC2F for a Legendrian $L \subset J^1M$ with cone points has the additions:

Let  $ R_\nu \subset M \setminus \Sigma_{\mathcal{C}}$ be a region that borders a cone point between sheets $S_k$ and $S_{k+1}$.  Then,  
\begin{enumerate}
\item  $\langle d_\nu S_{k+1} , S_k \rangle =0$, and 
\item for any $i<k$ (resp. $k+1<j$), there are $\langle d_{\nu} S_{k+1}, S_i \rangle$  $(i,k)$-handleslide arcs (resp. $\langle d_{\nu} S_{j}, S_k \rangle$ $(k+1,j)$-handleslide arcs) with endpoints at the cone point. 
\end{enumerate}
\end{remark}

\subsection{An example obstructing a trival bundle domain}

To illustrate the obstruction from Proposition \ref{prop:ObstructA} (2), consider a non-seperating curve $\gamma \subset T^2$.  There is a corresponding Legendrian $L_\gamma \subset J^1T^2$ with $2$-sheeted front projection having a crossing arc above $\gamma$ and no other crossings or cusps.

\begin{proposition}
There is no tame generating family $F: E \rightarrow \R$ for $L_\gamma$ whose domain is a trivial bundle over $T^2$.  
\end{proposition}

\begin{proof}
To apply Proposition \ref{prop:ObstructA}, we must show that any augmentation $\epsilon$ has non-trivial monodromy representation, $\Phi_{\epsilon,x_0}$.  Let $\mathcal{C}$ be an MC2F that agrees with the corresponding CHD, $\mathcal{D} \leftrightarrow \epsilon$, on the $1$-skeleton.   Take $x_0$ to be slightly shifted off of $\gamma$, and $\sigma$ a loop based at $x_0$ that intersects $\gamma$ geometrically once just before its endpoint.  The chain level continuation map for $\mathcal{C}$ has matrix of the form
\[
f(\sigma) = Q (I+E_{1,2})^{n} = \left[\begin{array}{cc} 0 & 1 \\ 1 & n \end{array}\right]
\] 
where $n$ is the number of handleslide arcs that $\sigma$ encounters and $Q$ is the permutation matrix for $(1\,2)$. 
The differential from $\mathcal{C}$ at $x_0$ vanishes (via Observation \ref{ob:CM2F} (2)), so we conclude that $f(\sigma)$ induces a non-identity map on homology, i.e. $\Phi_{\epsilon,x_0}([\sigma]) \neq 1$.   
\end{proof}

Note that $L_{\gamma}$ does have an obvious generating family whose domain is a non-trivial $2$-fold cover of $T^2$.

\end{document}